\documentclass{article}

\author{Dor Elboim and Ofir Gorodetsky} 
\date{}
\title{Multiplicative arithmetic functions and the generalized Ewens measure}

\newcommand{\Addresses}{{
		\bigskip
		\footnotesize
		
		\textsc{Department of Mathematics, Princeton University, Princeton, NJ 08544, USA}\par\nopagebreak
		\textit{E-mail address:} \texttt{delboim@math.princeton.edu}
		
		\medskip
		
		\textsc{Mathematical Institute, University of Oxford, Oxford, OX2 6GG, UK}\par\nopagebreak
		\textit{E-mail address:} \texttt{ofir.goro@gmail.com}
}}

\usepackage[margin=1in]{geometry}
\usepackage[all]{xy}
\usepackage{amssymb}
\usepackage{amsmath}
\usepackage{amsthm}
\usepackage{hyperref}
\usepackage{mathtools}
\usepackage{BOONDOX-cal}
\usepackage{url}
\usepackage{xcolor}
\usepackage{dsfont}

\theoremstyle{plain}

\newtheorem{thm}{Theorem}[section]

\newtheorem{lem}[thm]{Lemma}  
\newtheorem{prop}[thm]{Proposition}
\newtheorem{cor}[thm]{Corollary}

\theoremstyle{definition}
\newtheorem{definition}[thm]{Definition}

\theoremstyle{remark}
\newtheorem{remark}[thm]{Remark}

\newcommand{\PP}{\mathbb{P}}
\newcommand{\RR}{\mathbb{R}}
\newcommand{\CC}{\mathbb{C}}

\newcommand{\NN}{\mathbb{N}}
\newcommand{\EE}{\mathbb{E}}

\newcommand{\Primes}{\mathcal{P}}
\newcommand{\PolyConst}{K}

\newcommand{\largep}{\mathcal{p}}
\newcommand{\Ewensc}{{r}}
\newcommand{\logexp}{{a}}

\DeclareMathOperator\betadist{beta}
\DeclareMathOperator\gammadist{gamma}
\DeclareMathOperator\PD{PD}

\mathtoolsset{showonlyrefs}

\numberwithin{equation}{section}

\begin{document}

\maketitle
\begin{abstract}
Random integers, sampled uniformly from $[1,x]$, share similarities with random permutations, sampled uniformly from $S_n$. These similarities include the Erd\H{o}s--Kac theorem on the distribution of the number of prime factors of a random integer, and Billingsley's theorem on the largest prime factors of a random integer. In this paper we extend this analogy to non-uniform distributions.

Given a multiplicative function $\alpha \colon \NN \to \RR_{\ge 0}$, one may associate with it a measure on the integers in $[1,x]$, where $n$ is sampled with probability proportional to the value $\alpha(n)$. Analogously, given a sequence $\{ \theta_i\}_{i \ge 1}$ of non-negative reals, one may associate with it a measure on $S_n$ that assigns to a permutation a probability proportional to a product of weights over the cycles of the permutation. This measure is known as the generalized Ewens measure.

We study the case where the mean value of $\alpha$ over primes tends to some positive $\theta$, as well as the weights $\alpha(p) \approx (\log p)^{\gamma}$. In both cases, we obtain results in the integer setting which are in agreement with those in the permutation setting.
\end{abstract}
\section{Introduction}
The analogy between permutations and integers is a well-established one, see the surveys \cite[Ch.~1]{arratia2003} and \cite{granville2008anatomy}. The analogy leads to advancements both in permutations and in integers, see e.g.~\cite{granville2006,granville2007prime,eberhard2016,eberhard2017}. This analogy always involves comparing a uniformly drawn integer in $[1,x]$ and a uniformly drawn permutation from $S_n$, where $n \approx \log x$. Our results suggest that the analogy persists even when the chosen measures are not  uniform. 

We begin with setup and notation.
Let $S_n$ be the symmetric group on $\{1,2,\ldots,n\}$. Given $\pi \in S_n$, we denote by $\ell_1(\pi) \ge \ell_2(\pi) \ge \ldots $ the lengths of the disjoint cycles of $\pi$,  arranged in non-increasing order. They satisfy
\begin{equation}\label{eq:SumCycles}
	\ell_1(\pi)+\ell_2(\pi)+\ldots =n. 
\end{equation}
We let $C_i(\pi)$ be the number of cycles of $\pi$ of length $i$ and denote by $C(\pi)$ the number of cycles in $\pi$.

In recent years, there has been significant activity in the study of permutations sampled according to \emph{cycle weights} \cite{yakymiv2007,timashev2008,lugo2009,betz2011,maples2012,nikeghbali2013,nikeghbali20132,ercolani2014,cipriani2015,dereich2015,storm2015,robles2018random}; this model is related to the study of the quantum
Bose gas in statistical mechanics, see e.g.~\cite{betz2009spatial,betz2011spatial,elboim2019}. To state the model, let $\theta_1,\ldots,\theta_n$ be non-negative reals (not all zero). 
The probability of a permutation $\pi$ with respect to the weights $\theta_i$  is defined to be
\begin{equation}\label{eq:GeneralizedEwens}
	\PP_{n,\theta_i}(\pi) = \frac{1}{h_n n!} \prod_{i=1}^{n} \theta_i^{C_i(\pi)} = \frac{1}{h_n n!} \prod_{\substack{c \in \pi\\\text{cycle of }\pi}} \theta_{|c|}
\end{equation}
where $h_n$ is the normalization constant, known as the partition function, given by
\begin{equation}\label{eq:DefPartition}
	h_n = \frac{1}{n!}\sum_{\pi \in S_n} \prod_{i=1}^{n} \theta_i^{C_i(\pi)}.
\end{equation}
We let $\pi_{n,\theta_i}$ be the random permutation whose probability distribution is \eqref{eq:GeneralizedEwens}. The measure $\PP_{n,\theta_i}$ is called a generalized Ewens measure.

We now describe an analogous measure on the positive integers up to $x$, which is the main object of study in this paper.
Given a positive integer $m \in \NN$, denote by $\largep_1(m) \ge \largep_2(m) \ge \ldots$ the prime factors of $m$ (repeated according to their multiplicity), arranged in non-increasing order. We have
\begin{equation}
	\log \largep_1(m) + \log \largep_2(m) + \ldots  = \log m.
\end{equation}
We denote by $\Omega(m)$ the number of prime factors of $m$, counted with multiplicity. 
If $p^k \mid n$ and $p^{k+1} \nmid n$, we write $p^k \mid \mid n$. This $k$ is known as the multiplicity of $p$ in $n$, and is denoted $\nu_p(n)$.

A function $\alpha\colon \NN \to \CC$ is called multiplicative if $\alpha(1)=1$ and $\alpha(nm)=\alpha(n)\alpha(m)$ for every coprime $n,m \in \NN$. Given a non-negative multiplicative function $\alpha$, we define a measure on the positive integers up to $x$ by
\begin{equation}\label{eq:MultMeasure}
	\PP_{x,\alpha}(m) = \frac{1}{S(x)} \alpha(m)=  \frac{1}{S(x)} \prod_{p^k \mid \mid m}\alpha(p^{k})
\end{equation}
where the product is over (maximal) prime powers dividing $m$, and $S(x)$ is the normalization constant
\begin{equation}\label{eq:SumS}
	S(x) = \sum_{m \le x} \alpha(m).
\end{equation}
We let \[N_x=N_{x,\alpha}\]
be the random integer whose probability distribution is \eqref{eq:MultMeasure}. 
In this paper we consider two different families of multiplicative measures on the integers, and compare our results with corresponding generalized Ewens measures.
\subsection{Constant mean value}
We consider multiplicative functions $\alpha\colon \NN \to \RR_{\ge 0}$ satisfying the following two conditions for some $\theta>0$, $d >-1$, $\logexp\in (0,1)$, $\eta \in (0,1/2]$ and $\Ewensc \in (0,2)$:
\begin{align}
\label{eq:Assumption1}
&\text{(I) } \sum_{p \le x} \frac{\alpha(p) \log p}{p^d} = \theta x + O\left(\frac{x}{\log^{\logexp} x}\right), \\
\label{eq:Assumption2} &\text{(II) }\frac{\alpha(p)}{p^d}=O(p^{1/2-\eta}), \quad \frac{\alpha(p^k)}{p^{dk}} = O(\Ewensc^k) \mbox{ for all $k \ge 2$.}
\end{align}
 Here $p$ denotes a prime number.  Recall that the prime number theorem says that $\sum_{p \le x} \log p \sim x$. Thus, \eqref{eq:Assumption1} should be interpreted as $\alpha(p)/p^d$ being, on average, of size $\theta$, and it is a common condition in multiplicative number theory. Condition \eqref{eq:Assumption2} is of a more technical nature. We did not strive to find the most general conditions for our theorem to hold, but rather to find conditions which are easy to work with, lead to short proofs and are satisfied for natural examples. Our result for these weights is the following.
\begin{thm}
	\label{thm:EKDG}
	Let $\alpha\colon \NN \to \RR_{\ge 0}$ be a multiplicative function satisfying \eqref{eq:Assumption1}--\eqref{eq:Assumption2} with $\theta>0$,	$d >-1$, $\logexp\in (0,1)$, $\eta \in (0,1/2]$ and $\Ewensc \in (0,2)$. As $x \to \infty$ we have
		\begin{equation}\label{eq:norm conv}
		\frac{ \Omega (N_{x})-\theta \log \log x }{\sqrt{\theta \log \log x }} \overset{d}{\longrightarrow} N(0,1)
		\end{equation}
and
		\begin{equation}\label{eq:pd conv}
		\left( \frac{\log \largep_1(N_x)}{\log x},  \frac{\log \largep_2(N_x)}{\log x}, \ldots  \right) \overset{d}{\longrightarrow} \mathrm{PD}(\theta),
		\end{equation}
		where $\mathrm{PD}(\theta)$ is the Poisson--Dirichlet distribution with parameter $\theta$ (defined in \S\ref{sec:prob}).
\end{thm}
Here the arrows indicate convergence in distribution. The proof of the first part of Theorem~\ref{thm:EKDG} applies to $\omega(N_x)$ as well, where $\omega(n)$ counts the number of prime factors of $n$ \emph{without} multiplicities, and the result is the same. The prototypical example of a function $\alpha$ satisfying the conditions is $\theta^{\omega(n)}$ (with $d=0$, $\eta=1/2$, $\Ewensc=1$ and any $\logexp>0$).

The case $\alpha \equiv 1$ of \eqref{eq:norm conv} is the Erd\H{o}s--Kac theorem \cite{erdos1940}. Our proof of it is a generalization of a proof given by Billingsley \cite{billingsley1969} to the original Erd\H{o}s--Kac theorem with $\alpha \equiv 1$. The case $\alpha \equiv 1$ of \eqref{eq:pd conv} is Billingsley's theorem \cite{billingsley1972}. Our proof of it is a generalization of Donnelly and Grimmett's proof \cite{donnelly1993}, who elucidated Billingsley's result. 

The Ewens measure with parameter $\theta$ ($>0$) is a measure on $S_n$, which may be defined by taking $\theta_i=\theta$ in the definition of the generalized Ewens measure. The partition function is $\binom{n+\theta-1}{n}$ in this case.
This measure has first appeared in the study of population genetics \cite{ewens1972}. The Ewens measure has found many practical applications, through its connection with Kingman's coalescent process \cite{kingman1982} and its occurrence in non-parametric Bayesian statistics \cite{antoniak1974}. 

Theorem~\ref{thm:EKDG} should be compared with two results on the Ewens measure, one on $C(\pi_{n,\theta})$ (the number of cycles in $\pi_{n,\theta}$) by Hansen \cite{hansen1990}, and another on $(\ell_1(\pi_{n,\theta})/n,\ell_2(\pi_{n,\theta})/n,\ldots)$ by Watterson \cite{watterson1976}.
\begin{thm}[Hansen, Watterson]\label{thm:ewens} Let $\theta>0$. As $n \to \infty$ we have
		\begin{equation}\label{eq:hansen}
		\frac{C(\pi_{n,\theta})-\theta\log n}{\sqrt{\theta\log n}} \overset{d}{\longrightarrow} N(0,1)
		\end{equation}
and
		\begin{equation}\label{eq:water}
		\left( \frac{ \ell_1(\pi_{n,\theta})}{n},  \frac{ \ell_2(\pi_{n,\theta})}{n}, \ldots  \right) \overset{d}{\longrightarrow} \mathrm{PD}(\theta).
		\end{equation}
\end{thm}
The first part of Theorem \ref{thm:ewens} was proven under more general conditions, e.g.~when $\sum_{i=1}^{n}\theta_i/n \to \theta$ sufficiently fast, see Lugo \cite{lugo2009} and the works of Manstavi\v{c}ius \cite{Manst2002,Manst2009,Manst2017}.

The similarity of Theorem~\ref{thm:EKDG} and Theorem~\ref{thm:ewens} is most apparent for functions $\alpha$ where $\alpha(p) \approx \theta$. It suggests an analogy between permutations chosen according to the Ewens measure and integers chosen according to multiplicative weights. We now discuss previous works.
\subsubsection{Erd\H{o}s--Kac}
In a series of works, Alladi \cite{alladi1982,alladi1984,alladi1985,alladi1122} proved a generalization of Erd\H{o}s--Kac involving weights $\alpha$ as well. His proof uses the combinatorial sieve and he requires $\alpha$ to satisfy a `level-of-distribution' condition which is not always easily verified. A related (but simpler) sieve-theoretic approach to Erd\H{o}s--Kac and its generalizations was introduced by Granville and Soundararajan \cite{granville2007}. This approach was used by Khan, Milinovich and Subedi to prove an Erd\H{o}s--Kac theorem with weights being $d_k$, the $k$th divisor function \cite{Khan2021}. 

See Elliott \cite{Elliott1,Elliott2} for a treatment of the Erd\H{o}s--Kac theorem with weights being the standard divisor function $d_2$ and its real powers. Tenenbaum proved in \cite[Cor.~2.5]{tenenbaum2017} an impressively general weighted Erd\H{o}s--Kac theorem, but  unlike Theorem~\ref{thm:EKDG}, he requires $\alpha(p)/p^d$ to be uniformly bounded. Both Elliott and Tenenbaum use characteristic functions and complex analysis while we avoid these.
\subsubsection{Billingsley}
Arratia, Kochman and Miller proved an analogue of Billingsley's theorem for \emph{normed arithmetic semigroups} satisfying certain growth conditions \cite[Thm.~2]{arratia2014}. A commutative semigroup $S$ is called normed arithmetic semigroup if it contains an identity element and admits unique factorization into `prime' elements. Furthermore, it should come equipped with a multiplicative norm function $s \mapsto |s| \in \RR_{>0}$, such that $N(x)=\#\{ s \in S: |s| \le x \}$ is a finite number for each $x >0$.
There is small overlap between \cite[Thm.~2]{arratia2014} and the second part of Theorem \ref{thm:EKDG} as there are multiplicative functions $\alpha$ satisfying \eqref{eq:Assumption1}--\eqref{eq:Assumption2} and coinciding with $\alpha_S(n):=\#\{ s \in S: |s| =n \}$ for some normed arithmetic semigroup $S$.
\subsection{Polynomially-growing weights}
In the permutation setting, the measure $\PP_{n,\theta_i}$ was studied extensively in the case of polynomially-growing cycle weights, that is
\begin{equation}\label{eq:polyweights}
\theta_n \approx A n^{\gamma},
\end{equation}
see \cite{erlihson2008,maples2012,ercolani2014,dereich2015,cipriani2015}. Ercolani and Ueltschi proved the following in \cite[Thm.~5.1]{ercolani2014}.
\begin{thm}[Ercolani and Ueltschi]\label{thm:ECPolyCyc}
	Let $\gamma>0$, and take
	\begin{equation}\label{eq:thetachoicepol}
	\theta_n=\frac{\Gamma(\gamma+n+1)}{n!}= (1+o(1) ) n^{\gamma}.
	\end{equation}
	As $n \to \infty$ we have
	\begin{equation}
	\EE C(\pi_{n,\theta_i}) \sim n^{\frac{\gamma}{\gamma+1}} \left(\frac{\Gamma(\gamma)}{\gamma^{\gamma}}\right)^{\frac{1}{\gamma+1}}.
	\end{equation}
\end{thm} 
The specific choice \eqref{eq:thetachoicepol} simplifies the computations greatly. Maples, Nikeghbali and Zeindler \cite[Cor.~1.2]{maples2012} were able to prove that $C(\pi_{n,\theta_i})$ converges, after an explicit normalization, to a normal distribution (for $\theta_n$ as in \eqref{eq:thetachoicepol} and also scalar multiples of it).

Next we describe a result of Ercolani and Ueltschi \cite[Thm.~5.1]{ercolani2014} about a permutation statistic which we have yet to discuss, $L_1(\pi)$. This is the length of the cycle of a permutation $\pi$ which contains the element $1$. In the Ewens case $\theta_i = \theta$, it is known that $L_1(\pi_{n,\theta_i})/n$ converges in distribution to a beta distribution \cite[\S6]{ercolani2014}. For polynomially-growing weights, Ercolani and Ueltschi proved that $L_1(\pi_{n,\theta_i})$ exhibits a very different behavior. First, the order of magnitude of $L_1(\pi_{n,\theta_i})$ in this case is $n^{1/(\gamma+1)}=o(n)$ and not $n$. Second, the limiting distribution is a gamma distribution, whose definition is recalled in \S\ref{sec:prob}.
\begin{thm}[Ercolani and Ueltschi]\label{thm:ECPolyL1}
	Let $\gamma>0$, and take $\theta_n$ as in \eqref{eq:thetachoicepol}. Then, as $n \to \infty$,
	\begin{equation}
	\frac{L_1(\pi_{n,\theta_i})}{n^{\frac{1}{\gamma+1}}} \overset{d}{\longrightarrow} \gammadist(\gamma+1,\Gamma(\gamma+1)^{1/(\gamma+1)}).
	\end{equation}
\end{thm}
See Dereich and M\"{o}rters \cite{dereich2015} for finer results about $L_1(\pi_{n,\theta_i})$ for similar weights.

We derive number-theoretic analogues of Theorems~\ref{thm:ECPolyCyc} and \ref{thm:ECPolyL1}. 
Since there is no such thing as `a prime divisor of $n$ containing a fixed element', we must turn to a different interpretation of $L_1(\pi)$.
\begin{definition}\label{def:sbs}
	Let $\mathbf{a} = \{a_j\}_{j \ge 1}$ be a sequence of non-negative reals summing to $0<S<\infty$. A \textbf{size-biased sampling} of an element from $\mathbf{a}$ is a random variable $X$ whose distribution is given by
	\begin{equation}
	\PP(X = a_j) = \frac{a_j \cdot \left| \left\{ i\ge 1 \ \big| \ a_i=a_j \right\} \right| }{S}.
	\end{equation}
\end{definition}
Suppose that $\PP$ is some conjugation-invariant measure on $S_n$ (e.g.~$\PP_{n,\theta_i}$). If $\pi\in S_n$ is sampled according to $\PP$, then the distribution of $L_1(\pi)$ coincides with the distribution of a typical cycle of $\pi$, that is: of a size-biased sampling of an element from $\{\ell_i(\pi)\}_{i \ge 1}$. See Lemma~\ref{lem:conjinv} below for the proof. 
It is now clear how to define an integer analogue of $L_1(\pi_{n,\theta_i})$: given $N_x$, we define $P_1(N_x)$ by letting $\log P_1(N_{x})$ be a size-biased sampling of an element from $\{ \log \largep_i(N_x) \}_{i \ge 1}$. We think of $P_1(N_x)$ as a typical prime divisor of $N_{x }$.

In the integer setting, the polynomial weights \eqref{eq:polyweights} correspond to 
\begin{equation}
\alpha(p) \approx \PolyConst \log^{\gamma} p.
\end{equation}
For our results, we require that for all primes $p$, 
\begin{align}
\label{eq:polycond1}
&\text{(I) } \alpha(p) = \PolyConst \log^{\gamma} p + O( \log^{-2} p ),\\
\label{eq:polycond2}
&\text{(II) } \sum_{k \ge 2} \frac{k\alpha(p^k)}{p^k} = O\bigg( \frac{1}{p \log^2 p} \bigg).
\end{align}
\begin{thm}\label{thm:integer_poly}
	Let $\alpha\colon \NN \to \RR_{\ge 0}$ be a multiplicative function satisfying \eqref{eq:polycond1}--\eqref{eq:polycond2} for some $K>0$, $\gamma>0$. As $x \to \infty$ we have
		\begin{equation}
		\frac{\log P_1(N_{x})}{(\log x)^{\frac{1}{\gamma+1}}} \overset{d}{\longrightarrow} \gammadist(\gamma+1,(K\Gamma(\gamma+1))^{1/(\gamma+1)})
		\end{equation}
and
		\begin{equation}\label{eq:poly exp omega}
		\EE \Omega(N_{x}) \sim (\log x)^{\frac{\gamma}{\gamma+1}}  \left( \frac{\PolyConst \Gamma(\gamma)}{\gamma^{\gamma}}\right)^{\frac{1}{\gamma+1}}.
		\end{equation}
\end{thm}
The proof of the second part of Theorem~\ref{thm:integer_poly} applies to $\omega(N_{x})$ as well. It is interesting that although the measure $\PP_{x,\alpha}$ assigns larger weights to larger primes, the typical prime factors of $N_x$ are much smaller than in the case of uniformly drawn integers between $1$ and $x$. Indeed, it follows from the first part of Theorem~\ref{thm:integer_poly} that $\log P_1(N_x) = o(\log x)$, while from Theorem~\ref{thm:EKDG} it follows that $\log P_1(N_x)= \Theta(\log x)$ for uniform integers. 

The main ingredient in the proof of Theorem~\ref{thm:integer_poly} is the asymptotics of $\sum_{n \le x} \alpha(n)$ for multiplicative $\alpha$ obeying \eqref{eq:polycond1}--\eqref{eq:polycond2}. This unusual sum was studied by Schwarz \cite{schwarz1965} and Marenich \cite{marenich1983}.
As they both appeal to the same Tauberian theorem \cite[Thm.~1]{ingham1941}, no error term is obtained. We prove the following estimate.
\begin{thm}\label{thm:PolyPartFunc}
	Let $\alpha\colon \NN \to \RR_{\ge 0}$ be a multiplicative function satisfying \eqref{eq:polycond1}--\eqref{eq:polycond2} for some $\PolyConst>0$, $\gamma>0$. There exists $\varepsilon>0$ for which
	\begin{equation}\label{eq:sum of alpha poly}
	\frac{1}{x}\sum_{n \le x} \alpha(n) = (1 + O((\log x)^{-\varepsilon} )) A_{\alpha}  \frac{\exp(B (\log x)^{\frac{\gamma}{\gamma+1}})}{(\log x) ^{\frac{\gamma+2}{2(\gamma+1)}}}
	\end{equation}
	as $x \to \infty$, where $A_{\alpha}$ is a positive constant and
	\begin{equation}\label{eq:B def}
	B=\left( 1+\frac{1}{\gamma }\right)  (\PolyConst  \Gamma (\gamma +1) )^{\frac{1}{\gamma +1}}.
	\end{equation}
\end{thm}
The proof of Theorem~\ref{thm:PolyPartFunc} shows that
	\begin{equation}\label{eq:poly consts}
	A_{\alpha} =A  \prod_{p} \frac{\sum_{k \ge 0} \frac{\alpha(p^k)}{p^k}}{\exp(\frac{K \log^{\gamma} p}{p})}
	\end{equation}
	for a constant $A$ depending only on $K$ and $\gamma$. Additionally, one may take $\varepsilon=1/(\gamma+1)$ if $\gamma>2$, and $\varepsilon$ arbitrarily close to $\gamma/(2(\gamma+1))$ otherwise.
\subsection*{Conventions}
In the arguments below, we think of the function $\alpha$ as fixed, and write $N_x$ for $N_{x,\alpha}$. We denote the set of prime numbers by $\Primes$, and reserve the letter $p$ for primes. The letters $C$ and $c$ always denote positive constants, which may vary from line to line. However, $C$ and $c$ depend only on the arithmetic function $\alpha$ considered unless otherwise stated. 
The arguments in the proofs always hold for sufficiently large $x$. The notation $A \gg 1$ indicates that $A$ is sufficiently large.

\subsection*{Acknowledgements}
The second author was supported by the European Research Council under the European Union's Horizon 2020 research and innovation
programme (grant agreements nos 786758 and 851318). We thank the anonymous referee for useful comments and suggestions.

\section{Preliminaries from probability theory}\label{sec:prob}
We denote by $\betadist (\alpha,\beta ) $ the beta distribution with shape parameters $\alpha $ and $\beta $ whose density with respect to Lebesgue measure on $[0,1]$ is given by
\[\frac{\Gamma (\alpha +\beta )x^{\alpha -1}\left(1-x\right)^{\beta -1}}{\Gamma (\alpha )\Gamma (\beta )},\quad x\in [0,1].\]
We denote by $\gammadist (\alpha,\beta ) $ the gamma distribution with shape parameters $\alpha $ and $\beta $ whose density with respect to Lebesgue measure on $[0,\infty)$ is given by
\[\frac{\beta^{\alpha}x^{\alpha-1}e^{-\beta x}}{\Gamma(\alpha)},\quad x\in [0,\infty).\]
We define the Poisson--Dirichlet distribution with parameter $\theta $, denoted by $\PD (\theta )$. Let $Y_1, Y_2,\dots $ be an i.i.d. sequence of random variables with $\betadist (1,\theta )$ distribution. Define the sequence 
\begin{equation}
Z_j=(1-Y_1)\cdots (1-Y_{j-1})Y_j,\quad j\ge 1.
\end{equation}
Intuitively $Z_1$ takes a $\betadist (1,\theta )$-distributed fraction of the unit interval. Conditioned on $Z_1$, $Z_2$ takes a $\betadist (1,\theta )$-distributed fraction of the remaining part of the interval, etc. Finally, the $\PD (\theta )$ distribution is defined to be the distribution of the sequence $(Z_j)_{j \ge 1}$ arranged in non-increasing order. 

In the proof of Theorem~\ref{thm:EKDG}, we establish the convergence in distribution to $\PD (\theta )$ by convergence of a certain sequence to a sequence of independent $\betadist (1,\theta )$ random variables. To this end we define the size-biased permutation of a sequence of random variables. Let $X_1,X_2,\dots $ be a non-increasing sequence of random variables such that 
\begin{equation}
\sum _{j=1}^{\infty} X_j=1,\quad \text{almost surely.}
\end{equation}
A sized-biased permutation $(\tilde{X_i})_{i}$ of the sequence $(X_i)_i$ is a random reordering of the elements of the sequence such that for any $j\ge 1$
\begin{equation}
\PP \left( \tilde{X} _1=X_j \ | \ X_1,X_2,\dots  \right)= X_j \cdot \left| \left\{ j'\ge 1 \ \big| \ X_{j'}=X_j \right\} \right|
\end{equation}
and inductively for $k\ge 1$
\begin{equation}
\PP \left( \tilde{X} _k=X_j \ | \ \tilde{X}_1,\dots ,\tilde{X}_{k-1},X_1,X_2,\dots  \right)= \frac{X_j \cdot \left( \left| \left\{ j'\ge 1 \ \big| \ X_{j'}=X_j \right\} \right|-\left| \left\{ j'<k  \ \big| \ \tilde{X}_{j'}=X_j \right\} \right|\right)}{1-\tilde{X}_1-\ldots -\tilde{X} _{k-1}}.
\end{equation}
The sized-biased permutation can be used to reconstruct the $\betadist (1,\theta )$ random variables from the $\PD (\theta )$ distribution in the following sense.
\begin{prop}\label{prop:sized-biased of PD}
	Let $X_1,X_2,\dots $ be a non-increasing sequence of random variables with 
	\begin{equation}
	\sum _{j=1}^{\infty} X_j=1,\quad \text{almost surely.}
	\end{equation}
	Then, the sequence $(X_1,X_2,\dots )$ has the $\PD (
	\theta )$ distribution if and only if the sequence 
	\begin{equation}
	\left( \frac{\tilde{X} _j}{1-\tilde{X}_1-\ldots -\tilde{X}_{j-1}} \right) _{j=1}^\infty
	\end{equation}
	is an i.i.d. sequence of $\betadist (1,\theta )$ random variables.
\end{prop}
For a discussion and references for Proposition~\ref{prop:sized-biased of PD} see the introduction of  \cite{pitman1997}. We use the following result in order to prove the convergence to $\PD (
\theta )$ in Theorem~\ref{thm:EKDG}.
\begin{prop}\label{prop:proveconpd}
	For any $n \ge 1$, let $X_1^{(n)},X_2^{(n)}, \dots $ be a non-increasing sequence of random variables with 
	\begin{equation}
	\sum _{j=1}^{\infty} X_j^{(n)}=1,\quad \text{almost surely.}
	\end{equation}
	Suppose that 
	\begin{equation}\label{eq:assumption on beta}
	\left( \frac{\tilde{X}^{(n)} _j}{1-\tilde{X}^{(n)}_1-\ldots -\tilde{X}^{(n)}_{j-1}} \right) _{j=1}^\infty \overset{d}{\longrightarrow} (Y_1,Y_2,\dots ), \quad n\to \infty,
	\end{equation}
	where $Y_1,Y_2,\dots $ is a sequence of i.i.d. $\betadist (1,\theta )$ random variables. Then, we have 
	\begin{equation}\label{eq:x1convpd}
	(X_1,X_2,\dots )\overset{d}{\longrightarrow} \PD (\theta ) , \quad n\to \infty. 
	\end{equation}
\end{prop}

\begin{proof}
	Consider the function $g\colon [0,1]^{\NN} \to [0,1]^{\NN}$  defined by 
	\begin{equation}
	g(y_1,y_2,\dots ) = \big ((1-y_1)\cdots (1-y_{k-1})y_{k} \big)_{k=1}^{\infty},
	\end{equation}
	and the function
	\begin{equation}
	r\colon \Big\{ (z_1,z_2,\dots ) \in [0,1]^{\NN } :  \sum _{j=1}^{\infty} z_j \le 1 \Big\} \to [0,1]^{\NN}
	\end{equation}
	that takes a sequence and returns the same sequence in non-increasing order. Since $r \circ g$ is continuous in the product topology on $[0,1]^{\NN}$ we have by \eqref{eq:assumption on beta} that
	\begin{equation}\label{eq:continuous mapping}
	(r\circ g) \left( \bigg( \frac{\tilde{X}^{(n)} _j}{1-\tilde{X}^{(n)}_1-\ldots -\tilde{X}^{(n)}_{j-1}} \bigg) _{j=1}^\infty \right) \overset{d}{\longrightarrow} (r \circ g) (Y_1,Y_2,\dots ), \quad n\to \infty.
	\end{equation}
By the definition of the Poisson--Dirichlet distribution  $(r \circ g) (Y_1,Y_2,\dots ) \sim \PD (\theta )$ and moreover, since  
	\begin{equation}
	g^{-1} (z_1,z_2,\dots ) =\left(\frac{z_j}{1-z_1-\ldots -z_{j-1}} \right)_{j=1}^{\infty},
	\end{equation}
	we have that 
	\begin{equation}
	(r\circ g) \left( \bigg( \frac{\tilde{X}^{(n)} _j}{1-\tilde{X}^{(n)}_1-\ldots -\tilde{X}^{(n)}_{j-1}} \bigg) _{j=1}^\infty \right) =r (\tilde{X}_1, \tilde{X}_2, \ldots )=(X_1,X_2,\dots ).
	\end{equation}
	Thus, \eqref{eq:continuous mapping} simplifies to \eqref{eq:x1convpd}, as needed.
\end{proof}

\begin{lem}\label{lem:conjinv}
	Let $\PP$ be a conjugation-invariant measure on $S_n$. Given $\pi \in S_n$, let $L_1(\pi)$ be the size of the cycle containing $1$, and let $\mathrm{Typ}(\pi)$ be a random variable which is a size-biased sampling of a cycle of $\pi$, according to Definition~\ref{def:sbs}. Then $L_1(\pi)$ and $\mathrm{Typ}(\pi)$ have the same distribution, where $\pi$ is a permutation drawn according to $\PP$.
\end{lem}
\begin{proof}
	By using the conjugation-invariance, for any $k \in \NN$ we have
	\begin{equation}
	\PP\left(L_1(\pi) = k\right) = \frac{1}{n}\sum_{i=1}^{n} \PP\left(L_i(\pi) = k \right)
	\end{equation}
	where $L_i(\pi)$ is the size of the cycle of $\pi$ containing $i$. Letting $U_n$ be a uniformly drawn integer from $\{1,2,\ldots,n\}$, we have just shown that
	\begin{equation}
	\PP\left(L_1(\pi) = k\right) = \PP\left(L_{U_n}(\pi) = k \right).
	\end{equation}
	The size of the cycle of $\pi$ which contains $U_n$ is a size-biased sampling of a cycle of $\pi$, which concludes the proof.
\end{proof}

\section{Asymptotically Ewens measure}
\subsection{Multiplicative number theory}
 \begin{thm}[La Bret{\`e}che and Tenenbaum \cite{de2021remarks}]\label{thm:wirsing}
Suppose a multiplicative function $\alpha\colon \NN\to\mathbb{R}_{\ge 0}$ satisfies
\begin{equation}
\sum_{p \le x} \alpha(p)\log p =\theta x + O\left( \frac{x}{\log ^{\logexp}x}\right), \qquad x \to \infty,
\end{equation}
for some $\theta>0$ and $\logexp \in (0,1)$. Suppose further that
\begin{equation}\label{eq:powerscond} \sum_{p} \bigg( \frac{\alpha(p)^2}{p^{2\sigma}}+\sum_{\nu \ge 2} \frac{\alpha(p^{\nu})}{p^{\nu \sigma}}\bigg) < \infty
\end{equation}
for some $\sigma \in (0,1)$. Then, for all sufficiently large $x$,
\begin{equation}
\sum_{n \le x} \alpha(n) =  ( 1 + O(\log ^{-\logexp} x)) \Gamma(\theta)^{-1} \prod_{p} \bigg( \sum_{i \ge 0} \frac{\alpha(p^i)}{p^{i}}\bigg) \left( 1-\frac{1}{p}\right)^{\theta} x (\log x)^{\theta-1}.  
\end{equation}
\end{thm}
\begin{cor}\label{cor:AsympSum}
Suppose a multiplicative function $\alpha\colon \NN \to \RR_{\ge 0}$ satisfies \eqref{eq:Assumption1}--\eqref{eq:Assumption2} for some $\theta>0$,	$d >-1$, $\logexp\in (0,1)$, $\eta \in (0,1/2]$ and $\Ewensc \in (0,2)$. Then, for all sufficiently large $x$,
\begin{equation}\label{eq:SumAlpha}
\sum_{n \le x} \alpha(n) = ( 1+O( \log^{-\logexp}x))A_{\alpha} x^{d+1}\log^{\theta-1}x
\end{equation}
where $A_{\alpha} =(d+1)^{-1}\Gamma(\theta)^{-1}\prod_{p}( \sum_{i \ge 0} \alpha(p^i)/p^{i(d+1)} )(1-1/p)^{\theta}>0$.
\end{cor}
\begin{proof}
Let $\alpha_{d}(n)=\alpha(n)/n^d$. This function is still multiplicative. If $\alpha$ satisfies \eqref{eq:Assumption1}--\eqref{eq:Assumption2}, then $\alpha_{d}$ satisfies these conditions as well with the same parameters except that the parameter $d$ is now $0$.

Evidently, the conditions of Theorem~\ref{thm:wirsing} hold for $\alpha_d$, with the same $\theta$ and $\logexp$, and with any $\sigma \in (0,1)$ that satisfies $\sigma > \max\{\log \Ewensc / \log 2,1-\eta\}$. Thus, we conclude from Theorem~\ref{thm:wirsing} that
\begin{equation}\label{eq:adsum}
\sum _{n\le x} \alpha_{d}(n)=( 1+ O( \log^{-\logexp} x))\frac{1}{\Gamma(\theta)} \prod_{p} \bigg(\sum_{i \ge 0} \frac{\alpha(p^i)}{p^{i(d+1)}}\bigg)\left(1-\frac{1}{p}\right)^{\theta} x \log^{\theta-1}x 
\end{equation}
as $x \to \infty$. Using integration by parts and \eqref{eq:adsum} we have
\begin{equation}
\begin{split}
\sum _{n \le x} \alpha(n)&= \sum _{n \le x} n^d \alpha_d(n) = x^d \bigg( \sum _{n\le x} \alpha_d(n) \bigg)-\int_{1}^{x} \bigg( \sum _{n \le t} \alpha_d(n) \bigg) (t^d)'\, \mathrm{d}t  \\
&=( 1 + O( \log^{-\logexp} x))A_{\alpha_d} x^{d+1} \log ^ {\theta -1} x -\int_{2}^{x} A_{\alpha_d} ( 1 + O( \log^{-\logexp} t))(t \log ^ {\theta -1} t)(dt^{d-1})\, \mathrm{d}t,
\end{split}
\end{equation}
which gives \eqref{eq:SumAlpha} with $A_{\alpha} = A_{\alpha_d}/(d+1)$.
\end{proof}
\begin{lem}\label{lem:bound on partial sum_p}
Suppose a multiplicative function $\alpha\colon \NN \to \RR_{\ge 0}$ satisfies \eqref{eq:Assumption1}--\eqref{eq:Assumption2} for some $\theta>0$,	$d >-1$, $\logexp\in (0,1)$, $\eta \in (0,1/2]$ and $\Ewensc \in (0,2)$. Let $p$ be a prime. We have
	\begin{equation}\label{eq:sum_pdiv1}
	\sum _{\substack{n \le x \\ p \mid n}}   \alpha(n)\le \frac{C(1+\alpha(p)p^{-d})}{p} x^{d+1} (\log x)^{\max\{\theta-1,0\}}.
	\end{equation}
	If furthermore $p \le \sqrt{x}$ we have
	\begin{equation}\label{eq:sum_pdiv2}
	\sum _{\substack{n \le x \\ p \mid n}}   \alpha(n)\le \frac{C(1+\alpha(p)p^{-d})}{p} x^{d+1} (\log x)^{\theta-1}.
	\end{equation}
\end{lem}
\begin{proof}
By multiplicativity of $\alpha$, we have
\begin{equation}\label{eq:pdiv_sum}
\sum _{\substack{n \le x \\ p \mid n}}  \alpha(n)=\sum _{k=1}^{\lfloor \log_p x \rfloor} \sum _{\substack{n \le x \\ p^k \mid \mid n }}  \alpha(n) \le \sum _{k=1}^{\lfloor \log_p x \rfloor} \alpha(p^k) \sum _{ \substack{l=1 \\ p \nmid l} }  ^{ \left\lfloor \frac{x}{p^k} \right\rfloor  }  \alpha(l) \le \sum _{k=1}^{\lfloor \log_p x \rfloor} \alpha(p^k) \sum _{ l=1 }  ^{\left\lfloor \frac{x}{p^k} \right\rfloor}  \alpha(l).
\end{equation}	
By Corollary~\ref{cor:AsympSum} there exists a constant $C$ such that
	\begin{equation}\label{eq:part_sum}
	\sum _{n\le y} \alpha(n) \le C   y^{d+1} \log ^{\theta -1} y \le C   y^{d+1} (\log (y+1))^{\max\{\theta-1,0\} }
	\end{equation}
	for all $y \ge 1$ (one needs to consider $y \in [1,2)$ separately from $y \ge 2$). Note that the right-hand side of \eqref{eq:part_sum} is monotone increasing in $y$ and so is $(\log (y+1))^{\max\{\theta-1,0\}}$. By \eqref{eq:pdiv_sum} and \eqref{eq:part_sum} we have
	\begin{equation}\label{eq:alphasumk}
	\sum _{\substack{n \le x \\ p \mid n}}  \alpha(n)  \le  C\sum _{k=1}^{\lfloor \log_p x \rfloor} \frac{\alpha(p^k)}{p^{k(d+1)}} x^{d+1} (\log  (x+1))^{\max\{\theta-1,0\}}.
	\end{equation}
From \eqref{eq:Assumption2} and \eqref{eq:alphasumk} we obtain
	\begin{equation}
		\sum _{\substack{n \le x \\ p\mid n}}  \alpha(n)  \le C x^{d+1} (\log (x+1))^{\max\{\theta-1,0\} } \bigg( \frac{\alpha(p)}{p^{d+1}}+\sum_{k \ge 2}  \big(\frac{\Ewensc}{p}\big)^k  \bigg)\le  \frac{C(1+\alpha(p)p^{-d})}{p} x^{d+1} (\log (x+1))^{\max\{\theta-1,0\}},
	\end{equation}
and so \eqref{eq:sum_pdiv1} holds.
To prove \eqref{eq:sum_pdiv2}, we split the right-hand side of \eqref{eq:pdiv_sum} into two sums, one for terms with $p^{k}\le \sqrt{x}$ and another for the rest. Set $L:=\lfloor \log _p \sqrt{x} \rfloor +1 \ge 2$.  For $1 \le k <L$, we have $\log^{\theta-1}(x/p^k) \le C \log^{\theta-1} x$, and so Corollary~\ref{cor:AsympSum} gives
	\begin{equation}
	\sum _{ l=1 }  ^{\left\lfloor \frac{x}{p^k} \right\rfloor}  \alpha(l) \le C \left( \frac{x}{p^k }\right)^{d+1} \log ^{\theta -1} x.
	\end{equation}
From \eqref{eq:Assumption2} we obtain
	\begin{equation}\label{eq:smallk}
	\begin{split}
	\sum _{k=1}^{L-1} \alpha(p^k) \sum _{ l=1 }  ^{\left\lfloor \frac{x}{p^k} \right\rfloor}  \alpha(l) \le C x^{d+1} \log ^{\theta -1 }x \sum _{k=1}^{L-1} \frac{\alpha(p^k)}{p^{k(d+1)}} \le \frac{C(1+\alpha(p)p^{-d})}{p}  x^{d+1} \log ^{\theta -1 }x.
	\end{split}
	\end{equation}
	When $L \le k \le \lfloor \log_p x \rfloor$ we use the bound
	\begin{equation}
	\sum _{ l=1 }  ^{\left\lfloor \frac{x}{p^k} \right\rfloor}  \alpha(l) \le C \left( \frac{x}{p^k }\right)^{d+1} \log ^{\theta } x,
	\end{equation}
which follows from \eqref{eq:part_sum},	to obtain, from \eqref{eq:Assumption2} again, that
	\begin{equation}\label{eq:bigk}
	\begin{split}
	\sum _{k=L}^{\lfloor \log_p x \rfloor} \alpha(p^k) \sum _{ l=1 }  ^{\left\lfloor \frac{x}{p^k} \right\rfloor}  \alpha(l) &\le  C x^{d+1} \log ^{\theta  } x\sum _{k=L}^\infty \frac{\alpha(p^k)}{p^{k(d+1)}} \le  C x^{d+1} \log ^{\theta  }x\sum _{k=L}^{\infty}\left(\frac{\Ewensc}{p}\right)^k \le C x^{d+1} \log ^{\theta  }x \left( \frac{\Ewensc}{p} \right)^L \\
	&\le \frac{C}{p} x^{d+1} \log ^{\theta  }x  \left( \frac{\Ewensc}{p} \right)^{L/2} \le \frac{C}{p} x^{d+1} \log ^{\theta  }x \left( \frac{\Ewensc}{p} \right)^{(\log_p \sqrt{x}) /2}     \le \frac{C}{p} x^{d+1} \log ^{\theta -1 }x.
	\end{split}
	\end{equation}
From \eqref{eq:smallk} and \eqref{eq:bigk} we obtain \eqref{eq:sum_pdiv2}.
\end{proof}
\begin{lem}\label{lem:sum on primes}
	Let $\alpha \colon\Primes \to \RR_{\ge 0}$. Fix $\theta \in \RR$ and define a function $E$ via
	\[\sum_{p \le x} \alpha(p) \log p = \theta x + E(x).\]
	Let $I \subseteq [2,\infty)$ be a finite closed interval and $g\colon I \to \RR$ be a differentiable function. Then
	\begin{equation}
		\left| \sum _{p \in I} \alpha(p)g(p)-\theta \int_{I} \frac{g(t)}{\log t } \,\mathrm{d}t  \right| \le \max_{t \in I} |g(t)\alpha(t)|+2\max_{t\in I} \left|\frac{g(t)}{\log t}\right|  |E(t)| +\int_{I} \left| \left(\frac{g(t)}{\log t}\right)'\right| |E(t)| \mathrm{d}t .
	\end{equation}
\end{lem}
\begin{proof}
We write $I=[x,y]$. We shall work with the half-open interval $(x,y]$; we may do so because the contribution of $[x,y] \setminus (x,y] = \{x\}$ to $\sum_{p \in I} \alpha(p)g(p)$ is at most $\max_{t \in I} |g(t)\alpha(t)|$. Using integration by parts,
	\begin{equation}\label{eq:intg1}
		\begin{split}
			\theta \intop _x^y \frac{g(t)}{\log t}\, \mathrm{d}t=  \frac{\theta g(y)}{\log y}\intop _x^y 1 \, \mathrm{d}t -\intop _x^y \theta(t-x) \left( \frac{g(t)}{\log t} \right)^{\prime} \, \mathrm{d}t
			=\frac{\theta g(y)}{\log y}(y-x)-\intop _x^y \theta(t-x) \left( \frac{g(t)}{\log t} \right)^{\prime} \, \mathrm{d}t
		\end{split}
	\end{equation}
	and, by Abel's summation formula,
	\begin{equation}\label{eq:intg2}
		\begin{split}
			\sum _{x < p\le y } \alpha(p)g(p) & = \sum_{x < n \le y} \mathds{1}_{ \left\{ n \mbox{ is prime}  \right\} }  \alpha(n) \log n \cdot\frac{g(n)}{\log n}  \\
			&=\frac{g(y)}{\log y}\bigg(\sum _{x< p\le y } \alpha(p) \log p \bigg)-\intop_{x}^{y} \bigg( \sum _{x< p\le t } \alpha(p) \log p\bigg) \left( \frac{g(t)}{\log t} \right)^{\prime} \, \mathrm{d}t\\
			&= \frac{g(y)}{\log y} \left( \theta(y-x) + E(y)-E(x)\right)-\intop _x^{y} \bigg( \sum _{x < p\le t } \alpha(p)\log p\bigg) \left(\frac{g(t)}{\log t}\right)^{\prime}\, \mathrm{d}t
		\end{split}
	\end{equation}
	From \eqref{eq:intg1}, \eqref{eq:intg2} and the definition of $E$ we obtain
	\begin{align}
		\sum _{p \in (x,y]} \alpha(p)g(p)-\theta \intop_{x}^{y} \frac{g(t)}{\log t } \,\mathrm{d}t &=\frac{g(y)}{\log y}(E(y)-E(x)) - \intop_{x}^{y} \left(\frac{g(t)}{\log t}\right)' (E(t)-E(x))\, \mathrm{d}t\\
		&=\frac{g(y)}{\log y}\left(E(y)-E(x)\right) - \intop_{x}^{y} \left(\frac{g(t)}{\log t}\right)' E(t)\, \mathrm{d}t+E(x) \left(\frac{g(y)}{\log y} - \frac{g(x)}{\log x}\right)\\
		&=\frac{g(y)}{\log y} E(y)- \frac{g(x)}{\log x} E(x) -\intop_{x}^{y} \left(\frac{g(t)}{\log t}\right)' E(t) \, \mathrm{d}t,
	\end{align}
	and now the required estimate follows by the triangle inequality.
\end{proof}

\subsection{Proof of second part of Theorem~\ref{thm:EKDG}}
\subsubsection{Auxiliary lemma}

\begin{lem}\label{lem:lower bound on partial sum}
	Fix $k \ge 1$. Let $\delta \in (0,1)$ and suppose that $x$ is sufficiently large in terms of $\delta$. Let $p_1,\dots
	,p_k $ be distinct primes such that $p_1\cdots p_k \le x^{1-\delta }$ and $p_i\ge x^{\delta }$ for all $1\le i\le k$. 
	Suppose a multiplicative function $\alpha\colon \NN \to \RR_{\ge 0}$ satisfies \eqref{eq:Assumption1}--\eqref{eq:Assumption2} for some $\theta>0$,	$d >-1$, $\logexp\in (0,1)$, $\eta \in (0,1/2]$ and $\Ewensc \in (0,2)$. We have
	\begin{equation}
	\sum _{\substack{n= \lceil \delta x \rceil \\ p_1,\ldots ,p_k \mid \mid n}} ^x  \alpha(n) \ge (1-3 \delta ^{d+1})A_{\alpha}   \left( \prod _{i=1}^k \frac{\alpha(p_i)}{p_i^{d+1}} \right) x^{d+1} \log ^{\theta -1} \left(\frac{x}{p_1 \cdots p_k}\right),
	\end{equation}
	where $A_{\alpha}$ is the constant given in Corollary~\ref{cor:AsympSum}.
\end{lem}

\begin{proof}
	By the multiplicativity of $\alpha$,
	\begin{equation}\label{eq:lower bound}
	\sum _{\substack{n= \lceil \delta x \rceil \\ p_1,\ldots, p_k \mid \mid n}} ^x  \alpha(n)= \sum _{\substack{l=\lceil \frac{\delta x}{p_1 \cdots  p_k } \rceil  \\  \forall i \ p_i \nmid l}}^{\lfloor \frac{x}{p_1 \cdots  p_k } \rfloor } \alpha\left(l \cdot p_1 \cdots p_k \right) 
	=\alpha(p_1)\cdots \alpha(p_k)\sum _{\substack{l=\lceil \frac{\delta x}{p_1 \cdots  p_k } \rceil  \\  \forall i \ p_i \nmid l}}^{\lfloor \frac{x}{p_1 \cdots  p_k } \rfloor } \alpha(l).
	\end{equation}
	For the sum in the right-hand side we have the naive lower bound
	\begin{equation}\label{eq:NaiveLower}
	\sum _{\substack{l=\lceil \frac{\delta x}{p_1 \cdots  p_k } \rceil  \\  \forall i \ p_i \nmid l}}^{\lfloor \frac{x}{p_1 \cdots  p_k } \rfloor } \alpha(l) \ge \sum _{l=\lceil \frac{\delta x}{p_1 \cdots  p_k } \rceil}^{\lfloor \frac{x}{p_1 \cdots  p_k } \rfloor } \alpha(l) -\sum _{i=1}^k \sum _{\substack{l=1  \\   \ p_i | l}}^{\lfloor \frac{x}{p_1 \cdots  p_k } \rfloor } \alpha(l).
	\end{equation}
	The first sum in the right-hand side of \eqref{eq:NaiveLower} can be estimated by applying Corollary~\ref{cor:AsympSum} twice, and the second sum can be bounded from above by \eqref{eq:sum_pdiv1}. We thus obtain
	\begin{equation}
	\begin{split}
	\sum _{\substack{l=\lceil \frac{\delta x}{p_1 \cdots  p_k } \rceil  \\  \forall i \ p_i \nmid l}}^{\lfloor \frac{x}{p_1 \cdots  p_k } \rfloor } \alpha(l)  &\ge A_{\alpha}\left(1-\frac{5}{2}\delta ^{d+1}\right) \left(\frac{x}{p_1\cdots p_k}\right)^{d+1} \log ^{\theta -1} \left( \frac{x}{p_1 \cdots p_k} \right)\\
	&  \qquad -C_k\max_{1 \le i \le k}(p_i^{-1}+\alpha(p_i)p_i^{-d-1}) \left(\frac{x}{p_1\cdots p_k}\right)^{d+1} \log ^{\max\{\theta-1,0\}} \left( \frac{x}{p_1 \cdots p_k} \right)\\
	&\ge  A_{\alpha}\left(1-3\delta ^{d+1}\right) \left(\frac{x}{p_1\cdots p_k}\right)^{d+1} \log ^{\theta -1} \left( \frac{x}{p_1 \cdots p_k} \right)
	\end{split}
	\end{equation}
	for sufficiently large $x$, as needed. Here we made use of $\alpha(p_i)/p_i^{d+1} \ll p_i^{-1/2}$ (by \eqref{eq:Assumption2}) and $p_i \ge x^{\delta}$.
\end{proof}
\subsubsection{Conclusion of proof}
From Corollary~\ref{cor:AsympSum} it follows that for any given $\varepsilon \in (0,1)$,
\begin{equation}
\PP\left( \frac{\log N_x}{\log x} \le 1-\varepsilon\right)=\PP\left( N_x \le x^{1-\varepsilon}\right) = \frac{\sum_{n \le x^{1-\varepsilon}}\alpha(n)}{\sum_{n \le x} \alpha(n)} \sim x^{-\varepsilon(d+1)} (1-\varepsilon)^{\theta-1} \rightarrow 0
\end{equation}
as $x \to \infty$. Hence, $\log N_x / \log x$ tends in distribution to $1$. Thus, it suffices to prove that 
\begin{equation}\label{eq:xjdef}
\left(X_j \right)_{j=1}^{\infty }:=\left(\frac{\log \largep_j (N_x)}{\log N_x}\right)_{j=1} ^{\infty} \overset{d}{\longrightarrow } \mathrm{PD}(\theta)
\end{equation}
as $x \to \infty$. By Proposition~\ref{prop:proveconpd}, it suffices to prove the following proposition.
\begin{prop}\label{prop:9356}
We have
	\begin{equation}\label{eq:assumption on beta2}
	\left( B_j \right)_{j=1}^{\infty }:=	\left( \frac{\tilde{X}_j}{1-\tilde{X}_1-\ldots -\tilde{X}_{j-1}} \right) _{j=1}^\infty \overset{d}{\longrightarrow} (Y_1,Y_2,\dots )
\end{equation}
as $x \to \infty$, where $Y_1,Y_2,\dots $ is a sequence of i.i.d. $\betadist (1,\theta )$ random variables and where $( \tilde{X} _j )_{j=1}^{\infty }$ is a sized-biased permutation (defined in \S\ref{sec:prob}) of the sequence $(X_j )_{j=1}^{\infty }=\left(\frac{\log \largep_j (N_x)}{\log N_x}\right)_{j=1} ^{\infty}$.
\end{prop}
\begin{remark}
To connect this proposition with number-theoretic terms, we introduce a sequence $P_j$ of `typical prime divisors' of $N_x$, defined as $P_j={N_x}^{\tilde{X}_j}$. The asymptotic behavior of these typical primes might be of independent interest. It follows from the proposition that $P_j$ for $j\ge 1$ satisfy the following limit law
	\begin{equation}
		\left( \frac{\log P_1(N_x)}{\log x},  \frac{\log P_2(N_x)}{\log x}, \frac{\log P_3(N_x)}{\log x}, \ldots  \right) \overset{d}{\longrightarrow} \big(Y_1, (1-Y_1)Y_2, (1-Y_1-Y_2)Y_3, \dots \big),
		\end{equation}
		as $x \to \infty$, where once again, $Y_1,Y_2,\dots $ is a sequence of i.i.d. $\betadist (1,\theta )$.
\end{remark}
\begin{proof}
Fix $k\ge 1$ and $0<a_j <b_j<1$ for any $1\le j \le k$. 
By the Portmanteau Lemma it suffices to show that 
	\begin{equation}\label{eq:LowerBoundProb}
	\liminf _{x\to \infty } \PP \left( \forall j \  \ a_j\le B_j \le b_j \right) \ge \PP \left( \forall j \ \ a_j \le Y_j \le b_j \right).
	\end{equation}
We have that 
	\begin{equation}
	B_j=\frac{\tilde{X}_j}{1-\tilde{X}_1-\ldots -\tilde{X}_{j-1}}= \frac{\log (P_j)}{\log \left( \frac{N_x}{P_1\cdots P_{j-1}} \right)}
	\end{equation}
	and therefore 
	\begin{equation}
	\PP \left( \forall j \  \ a_j\le B_j \le b_j \right)=\sum _{n \le x}  \sum _{p_1,\dots ,p_k} \PP \left( N_x=n,\, \forall j \ \  P_j=p_j \right),
	\end{equation}
	where the inner sum is over a sequence of $k$ primes $p_1,\dots ,p_k$ such that for any $1\le j \le k$ we have   $ \left( n/(p_1\cdots p_{j-1}) \right)^{a_j} \le p_j\le \left( n/(p_1\cdots p_{j-1}) \right)^{b_j}$. Let $0<\delta < \min \{a_1,\dots a_k \} \prod _{i=1}^k(1-b_i)$ and put $x_j:=x/(p_1\cdots p_{j-1})$ for any $1 \le j\le k+1$. We have the following lower bound:
	\begin{equation}\label{eq:sum bj interval}
	\PP \left( \forall j \  \ a_j\le B_j \le b_j \right) \ge \sum _{\lceil \delta x \rceil  \le n \le x}  \sum _{\substack{p_1,\dots ,p_k \mbox{ distinct}\\x_j^{a_j}\le p_j \le (\delta x_j)^{b_j}, \ p_j \mid \mid n }} \PP \left( N_x=n,\,  \forall j \ \  P_j=p_j \right).
	\end{equation}
By the definition of a sized-biased permutation, when $p_1,\dots ,p_k $ are distinct and $p_j \mid \mid n$ we have 	
	\begin{equation}
	\PP \left( N_x=n, \, \forall j \ \  P_j=p_j \right)=\frac{\alpha(n)}{\sum _{m \le x} \alpha(m)} \prod _{j=1}^k \frac{ \log p_j}{\log \left( \frac{n}{p_1 \cdots p_{j-1}} \right)} \ge \frac{\alpha(n)}{\sum _{m \le x} \alpha(m)} \prod _{j=1}^k \frac{\log p_j}{\log  x_j}
	\end{equation}
	for any $n \le x$. Thus, changing the order of summation in \eqref{eq:sum bj interval} we obtain
	\begin{equation}\label{eq:change order of summation}
	\begin{split}
	\PP \left( \forall j \  \ a_j\le B_j \le b_j \right) & \ge \sum _{\substack{x_1 ^{a_1} \le p_1 \le (\delta x_1)^{b_1}  }}  \cdots  \sum _{\substack{x_k ^{a_k} \le p_k \le (\delta x_k)^{b_k}  \\ p_k \notin \{p_1,\dots p_{k-1}\}}} \sum _{ \substack{\lceil \delta x \rceil \le n \le x \\ p_i \mid \mid n} } \frac{\alpha(n)}{\sum _{m \le x} \alpha(m)} \prod _{j=1}^k \frac{\log p_j}{\log  x_j} \\
	& \ge \frac{1}{\sum _{m \le x} \alpha(m)}\sum _{\substack{x_1 ^{a_1} \le p_1 \le (\delta x_1)^{b_1}}} \frac{\log p_1}{\log x_1} \cdots \!\! \sum _{\substack{x_k ^{a_k} \le p_k \le (\delta x_k)^{b_k}  \\ p_k \notin \{p_1,\dots p_{k-1}\}}} \frac{\log p_k}{\log x_k} \sum _{ \substack{\lceil \delta x \rceil \le n \le x \\ p_i \mid \mid n}}\alpha(n).
	\end{split}
	\end{equation}
	Next, we would like to use Lemma~\ref{lem:lower bound on partial sum} in order to lower bound the inner sum in \eqref{eq:change order of summation}. We have $x_j=p_j x_{j+1}\le (\delta x_j)^{b_j} x_{j+1}\le x_j^{b_j} x_{j+1}$ and therefore, for any $1\le j \le k$ we have $x_{j+1}\ge x_{j}^{1-b_j}\ge \cdots \ge  x^{\nu }$ where $\nu :=\prod _{j=1}^k (1-b_j)$. Thus $p_j\ge x^{a_j \nu}\ge x^\delta $ and moreover $p_1 \cdots p_k=x/x_{k+1}\le x^{1-\nu } \le x^{1-\delta }$. We get that the assumptions of Lemma~\ref{lem:lower bound on partial sum} hold and therefore using also Corollary~\ref{cor:AsympSum} we obtain
	\begin{equation}\label{eq:using the lemma}
	\begin{split}
	& \PP \left( \forall j \  \ a_j\le B_j \le b_j \right) 
	\ge \frac{1-4\delta ^{d+1}}{\log ^{\theta -1}x }\sum _{\substack{x_1 ^{a_1} \le p_1 \le (\delta x_1)^{b_1}  }} \frac{\alpha(p_1) \log p_1}{p_1^{d+1} \log x_1} \cdots \!\! \sum _{\substack{x_k ^{a_k} \le p_k \le (\delta x_k)^{b_k}   \\ p_k \notin \{p_1,\dots p_{k-1}\}}} \frac{\alpha(p_k) \log p_k}{p_k^{d+1} \log x_k} \log ^{\theta -1} x_{k+1} \\
	&=\left( 1-4\delta ^{d+1} \right)\sum _{\substack{x_1 ^{a_1} \le p_1 \le (\delta x_1)^{b_1}  }} \frac{\alpha(p_1) \log p_1}{p_1 ^{d+1} \log x_1} \left( 1-\frac{\log p_1}{\log x_1} \right)^{\theta -1} \cdots \!\! \sum _{\substack{x_k ^{a_k} \le p_k \le (\delta x_k)^{b_k}  \\ p_k \notin \{p_1,\dots p_{k-1}\}}} \frac{\alpha(p_k) \log p_k}{p_k ^{d+1} \log x_k} \left( 1-\frac{\log p_k}{\log x_k} \right)^{\theta -1}
	\end{split}
	\end{equation}
	for sufficiently large $x$. Consider the innermost sum. We define the function $g=g_{n_k}\colon [x_k^{a_k},(\delta x_k)^{b_k}]\to \RR$ by
	\begin{equation}
	g(t):=\frac{\log t }{t} \left(1-\frac{\log t}{\log x_k}\right)^{\theta -1}.
	\end{equation}
	We apply Lemma~\ref{lem:sum on primes} with $\alpha(n)/n^d$ and this $g$, obtaining
	\begin{equation}
	\begin{split}
	\sum _{\substack{x_k ^{a_k} \le p_k \le (\delta x_k)^{b_k}  \\ p_k \notin \{p_1,\dots p_{k-1}\}}} \frac{\alpha(p_k) \log p_k}{p_k ^{d+1} \log x_k} \left( 1-\frac{\log p_k}{\log x_k} \right)^{\theta -1}=O_{\delta}\left(\frac{1}{\log x}\right) + \frac{1}{\log x_k}\sum _{\substack{x_k ^{a_k} \le p_k \le (\delta x_k)^{b_k}  }} \frac{\alpha(p_k)}{p_k^ d} g(p_k)  \\
	=O_{\delta}\left(\frac{1}{\log x}\right) + \frac{1}{ \log x_k} \intop _{x_k^{a_k}} ^{(\delta x_k)^{b_k}} \frac{\theta }{t} \left( 1-\frac{\log t}{\log x_k} \right)^{\theta -1} \, \mathrm{d}t= O_{\delta}\left(\frac{1}{\log x}\right)+  \intop_{a_k} ^{b_k} \theta (1-y)^{\theta -1} \, \mathrm{d}y
	\end{split}
	\end{equation}
	where in the last equality we performed the change of variables $y=\log t / \log x_k$. Substituting the last estimate into \eqref{eq:using the lemma} we get the same expression with $k$ replaced by $k-1$. Thus, for sufficiently large $x$ depending on $\delta $,
	\begin{equation}
	\PP \left( \forall j \  \ a_j\le B_j \le b_j \right) \ge (1-5\delta^{d+1}) \prod _{j=1}^k \intop _{a_j} ^{b_j} \theta (1-y)^{\theta -1} \, \mathrm{d}y=(1-5\delta^{d+1})  \PP \left( \forall j \ a_j\le Y_j \le b_j \right),
	\end{equation}
and so
	\begin{equation}
	\liminf _{x\to \infty } \PP \left( \forall j \  \ a_j\le B_j \le b_j \right) \ge (1-5\delta^{d+1})  \PP \left( \forall j \ \ a_j \le Y_j \le b_j \right).
	\end{equation}
	Since $\delta $ is arbitrary it follows that \eqref{eq:LowerBoundProb} holds, as needed. 
\end{proof}

\subsection{Proof of first part of Theorem~\ref{thm:EKDG}}
\subsubsection{Preparatory results}\label{sec:auxKac}
We need the following results from probability, which are given as Remarks 1, 2 and 3 in \cite{billingsley1969}.
\begin{lem}\label{lem:zero exp}
	Let $D_x$, $E_x$ be random variables defined for any $x \gg 1$.
	\begin{enumerate}
		\item If $D_x \overset{d}{\longrightarrow} 1$ and $E_x \overset{d}{\longrightarrow} 0$, then $U_x \overset{d}{\longrightarrow} N(0,1)$ if and only if $D_x U_x + E_x \overset{d}{\longrightarrow} N(0,1)$.
		\item Let $X \sim N(0,1)$. If $\EE D^k_x \to \EE X^k$ for each $k \ge 1$ then $D_x \overset{d}{\longrightarrow} N(0,1)$.
		\item Let $X \sim N(0,1)$. If $D_x \overset{d}{\longrightarrow} N(0,1)$ and if $\sup_{x} \EE |D_x|^{k+\varepsilon} < \infty$ for some $\varepsilon>0$ then $\EE D_x^k \to \EE X^k$.
	\end{enumerate}
\end{lem}
Recall that $\Primes$ is the set of primes. By definition,
\begin{equation}\label{eq:OmegAsSum}
\omega (N_x)=\sum _{p\in \Primes}  \mathds{1} _{ \{p \mid N_x \} }.
\end{equation}
\begin{lem}\label{lem:Omega omega}
We have $\EE\left| \Omega(N_x)- \omega(N_x) \right| = O(1)$.
\end{lem}
\begin{proof}
Consider the multiplicative function $\widetilde{\alpha}(n) := \alpha(n) t^{\Omega(n)-\omega(n)}$, where we choose $1<t<2/\Ewensc$, so that $\widetilde{\alpha}$ will still satisfy \eqref{eq:Assumption1}--\eqref{eq:Assumption2} with the same parameters, except $\Ewensc$ replaced with $\Ewensc  t$.
Applying Corollary~\ref{cor:AsympSum} with $\alpha$ and $\widetilde{\alpha}$, we obtain that
\begin{equation}
\EE t^{\Omega(N_x)-\omega(N_x)} = \frac{\sum_{n \le x} \widetilde{\alpha}(n)}{\sum_{n \le x} \alpha(n)} =\frac{A_{\widetilde{\alpha}}}{A_{\alpha}}(1+o(1))= O_t(1)
\end{equation}
which, by Jensen's inequality for instance, implies that $\EE\left| \Omega(N_x)- \omega(N_x) \right|$ is bounded as $x \to \infty$.
\end{proof}
For $x \gg 1$, we define the subset
\begin{equation}
\Primes _x:=\left\{ p \in \Primes : \  \log ^4 x  \le p\le \exp \left( \exp \left( \log \log x-\log ^{\frac{1}{3}} \log x \right) \right)   \right\}.
\end{equation}
For each prime $p \gg 1$, define a Bernoulli random variable $X_p$ such that $\PP(X_p = 1) = \alpha(p)/p^{d+1}$ (for $p \gg 1$, this is in $[0,1]$) and such that the different $X_p$s are independent. Define $\sigma_x$ by
\begin{equation} 
\sigma^2_x := \sum_{p \in \Primes_{x}} \frac{\alpha(p)}{p^{d+1}}\left(1 - \frac{\alpha(p)}{p^{d+1}} \right).
\end{equation}
We define the random variables
\begin{equation}
A_x := \frac{\omega (N_x)-\theta \log \log x}{\sqrt{\theta \log \log x }}, \quad B_x := \frac{\sum_{p \in \Primes _x} \left( \mathds{1} _{ \{p \mid N_x \} } -\frac{\alpha(p) }{p^{d+1}} \right)}{\sigma_x}, \quad C_x := \frac{\sum_{p \in \Primes _x} \left( X_p -\frac{\alpha(p) }{p^{d+1}} \right)}{\sigma_x}.
\end{equation}
\begin{lem}\label{lem:sum primes erdos}
We have $\sum_{p \in \Primes_{x}} \alpha(p)/p^{d+1} = \theta \log \log x + O(\log^{\frac{1}{3}} \log x)$ and $\sigma_x^2 = \theta \log \log x + O(\log^{\frac{1}{3}} \log x)$.
\end{lem}
\begin{proof}
Applying Lemma~\ref{lem:sum on primes} with $\alpha(p)/p^d$ in place of $\alpha$, $g(t)=1/t$ and the interval $[\log^4 x, \exp(\exp(\log \log x -\log^{1/3} \log x))]$, we find that $\sum_{p \in \Primes_{x}} \alpha(p)/p^{d+1} = \theta \log \log x + O(\log^{\frac{1}{3}} \log x)$.
Since $\sum_{p \in \Primes} \alpha^2(p)/p^{2(d+1)}=O(1)$, the estimate for $\sigma_x^2$ follows from the first estimate.
\end{proof}
\begin{lem}\label{lem:mom c bounded}
For each integer $k \ge 1$, we have $\sup_{x} |\EE C_x^k| < \infty$. In particular, $\sup_{x} \EE |C_x|^{2k} < \infty$.
\end{lem}
\begin{proof}
	Let $Y_p := X_p - \alpha(p)/p^{d+1}$.
	We have the following expansion:
	\begin{equation}\label{eq:the kth moment}
	\EE C_x^k = \frac{1}{\sigma_x^k} \sum _{p_1,\dots ,p_k \in \Primes _x} \EE \left[ Y_{p_1} \cdots Y_{p_k} \right]= \frac{1}{\sigma_x^k} \sum _{m=1}^k \sum _{ \substack{\  l_1,\dots , l_m \ge 1  \\ \sum l_i =k }  } \binom{k}{l_1, \dots , l_m}  S(l_1,\dots , l_m),
	\end{equation}
	where
	\begin{equation}
	S(l_1, \dots ,l_m):=\sum _{ \substack{q_1,\dots ,q_m \in \Primes _x \\ q_1<\cdots <q_m } } \EE \left[ Y_{q_1}^{l_1}\right] \EE \left[ Y_{q_2}^{l_2}\right]  \cdots \EE \left[ Y_{q_m}^{l_m} \right] .
	\end{equation}
	As we have $\EE Y_p=0$, it follows that $S(l_1,\dots, l_m)$ vanishes if $l_i=1$ for some $i$, and so we may restrict the summation in \eqref{eq:the kth moment} to $l_i \ge 2$. Since $|Y_p| \le 1$ (for $p \gg 1$), it follows that if $l_i \ge 2$ then $|\EE \left[ Y_{q_i}^{l_i} \right]| \le \EE \left[ Y_{q_i}^{2} \right]$, so that
	\begin{equation}
	S(l_1,\dots, l_m) \le \left( \sum_{p \in \Primes_{x}} \EE \left[ Y_p^2 \right] \right)^m = \sigma_x^{2m}. 
	\end{equation} 
	As $\sum_{i=1}^{m} l_i = k$ and $l_i \ge 2$, it follows that $2m \le k$. For $x \gg 1$ we have $\sigma_x \ge 1$ by Lemma~\ref{lem:sum primes erdos}, and so 
	\begin{equation}\label{eq:the kth moment2}
	\left| \EE C_x^k \right| \le   \sum _{m=1}^k \frac{\sigma_x^{2m}}{\sigma_x^k}\sum _{ \substack{\  l_1,\dots , l_m \ge 2  \\ \sum l_i =k }  } \binom{k}{l_1, \dots , l_m}  \le  \sum _{m=1}^k \sum _{ \substack{\  l_1,\dots , l_m \ge 2  \\ \sum l_i =k }  } \binom{k}{l_1, \dots , l_m} 
	\end{equation}
	as needed.
\end{proof}

\begin{lem}\label{lem:prob that q_1 q_m|N}
Fix $m \ge 1$. Let $q_1,\dots ,q_m \in \Primes _x$ be distinct primes. We have for $x \gg 1$
\begin{equation}
\left| \PP \left( \forall j, \  q_j \mid N_x \right)-\prod _{j=1}^m \frac{\alpha(q_j)}{q_j^{d+1}} \right|\le C_m \exp \left( -\log ^{\frac{1}{4}} \log x \right) \prod _{j=1}^m \frac{1+\alpha(q_j)q_j^{-d}}{q_j}.
\end{equation}
\end{lem}
\begin{proof}
By multiplicativity of $\alpha$ we have 
\begin{equation}\label{eq:alphaqsum}
\sum _{\substack{ n\le x \\ q_1\cdots q_m \mid n }} \alpha(n)=\sum _{l_1=1}^\infty \cdots \sum _{l_m=1}^{\infty } \sum_{ \substack{n\le x \\ \forall j, \ q_j^{l_j}\mid\mid n }}  \alpha(n)= \sum _{l_1=1}^\infty \cdots \sum _{l_m=1}^{\infty } \prod _{j=1}^m \alpha(q_j^{l_j}) \sum _{ \substack{l=1 \\  \forall j, \ q_j \nmid l }}^{ \lfloor \frac{x}{q_1^{l_1} \cdots  q_m^{l_m} } \rfloor }  \alpha(l).
\end{equation}
The term corresponding to $(l_1,\ldots ,l_m)=(1,\ldots ,1)$ in \eqref{eq:alphaqsum} may be estimated as
\begin{multline}\label{eq:li ones}
\left| \prod _{j=1}^m \alpha(q_j) \sum _{ \substack{l=1 \\  \forall j, \ q_j \nmid l }}^{ \lfloor \frac{x}{q_1 \cdots  q_m } \rfloor }  \alpha(l)- \prod _{j=1}^m \alpha(q_j) \sum _{ l=1 }^{ \lfloor \frac{x}{q_1 \cdots  q_m } \rfloor }  \alpha(l) \right| \le \prod _{j=1}^m \alpha(q_j) \sum _{k=1}^m \sum _{ \substack{l=1 \\   \ q_k \mid l }}^{ \lfloor \frac{x}{q_1 \cdots  q_m } \rfloor }  \alpha(l) \\
\le C x^{d+1}  \log ^{\theta -1} x \prod _{j=1}^m \frac{\alpha(q_j)}{q_j^{d+1}} \sum _{k=1}^m \frac{1+\alpha(q_k)q_k^{-d}}{q_k} \le C_m x^{d+1} \log^{\theta -3 } x \prod_{j=1}^{m} \frac{\alpha(q_j)}{q_j^{d+1}},
\end{multline}
where in the second inequality we used \eqref{eq:sum_pdiv2} and \eqref{eq:Assumption2}. Using Corollary~\ref{cor:AsympSum} we get
\begin{equation}\label{eq:main l_i ones}
\begin{split}
\prod _{j=1}^m \alpha(q_j) \sum _{ l=1 }^{ \lfloor \frac{x}{q_1 \cdots  q_m } \rfloor }  \alpha(l)&= (1+O(\log^{-a} x)) A_{\alpha}  x^{d+1} \log ^{\theta -1} \left(\frac{x}{q_1 \cdots q_m}\right) \prod _{j=1}^m \frac{\alpha(q_j)}{q_j^{d+1}} \\
&= \left( 1+O\left( \frac{\log^{\logexp} (q_1\cdots q_m)}{\log^{\logexp} x} \right)  \right) A_{\alpha} x^{d+1} \log ^{\theta -1} x \prod _{j=1}^m \frac{\alpha(q_j)}{q_j^{d+1}}\\
&=\left( 1+O_m\left( \exp \left( -\log ^{\frac{1}{4}} \log x \right)\right)\right) A_{\alpha} x^{d+1} \log ^{\theta -1} x \prod _{j=1}^m \frac{\alpha(q_j)}{q_j^{d+1}}.
\end{split}
\end{equation}
The contribution of terms with $(l_1,\dots ,l_m)\neq (1,\dots ,1)$ to \eqref{eq:alphaqsum} may be bounded using Corollary~\ref{cor:AsympSum} as follows:
\begin{equation}\label{eq:li non ones}
\begin{split}
\sum _{ (l_1,\dots ,l_m)\neq (1,\dots ,1) } & \prod _{j=1}^m \alpha(q_j^{l_j}) \sum _{ \substack{l=1 \\  \forall j, \ q_j \nmid l }}^{ \lfloor \frac{x}{q_1^{l_1} \cdots  q_m^{l_m} } \rfloor }  \alpha(l) \le C x^{d+1}\log ^{\theta } x \sum _{(l_1,\dots ,l_m)\neq (1,\dots ,1)} \prod _{j=1}^m \frac{\alpha(q_j^{l_j})}{q_j^{(d+1)l_j}}\\
&\le C_m x^{d+1} \log ^{\theta} x \sum _{j=1}^{m} \left( \sum _{l_j=2}^{\infty} \frac{\alpha(q_j^{l_j})}{q_j^{(d+1)l_j}} \right) \prod _{\substack{i=1 \\ i\neq j}}^{m} \left(\sum _{l_i=1}^{\infty} \frac{\alpha(q_i^{l_i})}{q_i^{(d+1)l_i}} \right) \\
&\le C_m x^{d+1} \log ^{\theta } x \sum _{j=1}^m \frac{1}{q_j^2} \prod _{\substack{i=1 \\ i\neq j}}^{m} \frac{1+\alpha(q_i)q_i^{-d}}{q_i} \le C_m  x^{d+1} \log ^{\theta -2} x  \prod _{j=1}^{m} \frac{1+\alpha(q_j)q_j^{-d}}{q_j}.
\end{split}
\end{equation}
To conclude the proof, we observe that $ \PP \left( \forall j, \  q_j \mid N_x \right) = \sum _{\substack{ n\le x \\ q_1\cdots q_m \mid n }} \alpha(n) / S(x)$, which combined with \eqref{eq:alphaqsum}, \eqref{eq:li ones}, \eqref{eq:main l_i ones}, \eqref{eq:li non ones} and Corollary~\ref{cor:AsympSum} gives the desired bound. 
\end{proof}
\begin{lem}\label{lem:dif b c mom}
	For each integer $k \ge 1$, we have $\EE B_x^k- \EE C_x^k  \to 0$.
\end{lem}
\begin{proof}
As in \eqref{eq:the kth moment}, we have
\begin{equation}\label{eq:the kth moment3}
\EE\bigg( \sum_{p \in \Primes_{x}} X_p \bigg)^k =  \sum _{m=1}^k \sum _{ \substack{\  l_1,\dots , l_m \ge 1  \\ \sum l_i =k }  } \binom{k}{l_1, \dots , l_m}  \sum _{ \substack{q_1,\dots ,q_m \in \Primes _x \\ q_1<\cdots <q_m } } \EE \left[ X_{q_1}^{l_1} X_{q_2}^{l_2} \cdots X_{q_m}^{l_m} \right].
\end{equation}
As $X_p$ assumes the values $1$ and $0$ only, we have $X_{q_i}^{l_i}=X_{q_i}$ and then $\EE [ X_{q_1}^{l_1} \cdots X_{q_m}^{l_m}] = \prod _{j=1}^m \frac{\alpha(q_j)}{q_j^{d+1}}$, so that
\begin{equation}\label{eq:the kth moment3_2}
\EE\bigg( \sum_{p \in \Primes_{x}} X_p \bigg)^k =  \sum _{m=1}^k \sum _{ \substack{\  l_1,\dots , l_m \ge 1  \\ \sum l_i =k }  } \binom{k}{l_1, \dots , l_m}  \sum _{ \substack{q_1,\dots ,q_m \in \Primes _x \\ q_1<\cdots <q_m } } \prod _{j=1}^m \frac{\alpha(q_j)}{q_j^{d+1}}.
\end{equation}
Similarly,
\begin{equation}\label{eq:the kth moment4}
\EE\bigg( \sum_{p \in \Primes_{x}} \mathds{1} _{ \{p \mid N_x \} }\bigg)^k =
\sum _{m=1}^k \sum _{ \substack{\  l_1,\dots , l_m \ge 1  \\ \sum l_i =k }  } \binom{k}{l_1, \dots , l_m}  \sum _{ \substack{q_1,\dots ,q_m \in \Primes _x \\ q_1<\cdots <q_m } }  \PP \left( \forall j, \  q_j \mid N_x \right).
\end{equation}
By Lemma~\ref{lem:prob that q_1 q_m|N} and Lemma~\ref{lem:sum on primes} with $g(t)=1/t$,
\begin{equation}
\begin{split}\label{eq:diff powers}
\bigg| \EE\bigg( \sum_{p \in \Primes_{x}} \mathds{1} _{ \{p \mid N_x \} }\bigg)^k& -  \EE\bigg( \sum_{p \in \Primes_{x}} X_p \bigg)^k \bigg| \\
&\le C_k \exp\left(-\log^{\frac{1}{4}} \log x \right) \sum _{m=1}^k \sum _{ \substack{\  l_1,\dots , l_m \ge 1  \\ \sum l_i =k }  } \binom{k}{l_1, \dots , l_m}  \sum _{ \substack{q_1,\dots ,q_m \in \Primes _x \\ q_1<\cdots <q_m } } \prod_{j=1}^{m} \frac{1+\alpha(q_j)q_j^{-d}}{q_j} \\
&\le C_k \exp\left(-\log^{\frac{1}{4}} \log x \right) \sum _{m=1}^k \sum _{ \substack{\  l_1,\dots , l_m \ge 1  \\ \sum l_i =k }  } \binom{k}{l_1, \dots , l_m}  \left( \sum _{p \le x} \frac{1+\alpha(p)p^{-d}}{p} \right)^m\\
& \le C_k \exp\left(-\log^{\frac{1}{4}} \log x \right) (\log \log x)^k.
\end{split}
\end{equation}
By the binomial theorem,  \eqref{eq:diff powers} and Lemma~\ref{lem:sum primes erdos}, we have
\begin{equation}
\begin{split}
\left| \EE  B_x^k- \EE C_x^k   \right| &\le \sigma_x^{-k} \sum_{i=0}^{k} \binom{k}{i}  \left| \EE\bigg(\sum_{p \in \Primes_{x}} \mathds{1} _{ \{p \mid N_x \} }\bigg)^{i} - \EE\bigg( \sum_{p \in \Primes_{x}} X_p \bigg)^{i} \right|  \left( \sum_{p \in \Primes_{x}} \frac{\alpha(p)}{p^{d+1}}\right)^{k-i}\\
&\le \sigma_x^{-k} C_{k}  \exp\left(-\log^{\frac{1}{4}} \log x \right) \sum_{i=0}^{k} \binom{k}{i}  (\log \log x)^{i} (\log \log x)^{k-i} \\
&\le  C_k(\log \log x)^{k/2}   \exp\left(-\log^{\frac{1}{4}} \log x \right) \to 0
\end{split}
\end{equation}
as needed.
\end{proof}
\subsubsection{Conclusion of proof}
By Lemma~\ref{lem:Omega omega} and the first part of Lemma~\ref{lem:zero exp} with $D_x = 1$ and $E_x = (\Omega(N_x)-\omega(N_x))/\sqrt{\theta \log \log x}$, it follows that \eqref{eq:norm conv} holds if and only if $A_x \overset{d}{\longrightarrow} N(0,1)$. Now let $D_x = \sqrt{\theta \log \log x}/\sigma _x$ and 
\begin{equation}\label{eq:ex}
E_x = \frac{\theta \log \log x - \sum_{p \in \Primes_x} \frac{\alpha(p)}{p^{d+1}}}{\sigma_x} - \frac{\sum_{p \in \Primes \setminus \Primes_x} \mathds{1}_{p \mid N_x}}{\sigma_x}.
\end{equation}
We have $D_x \overset{d}{\longrightarrow} 1$. Moreover, in the sum in the second fraction in \eqref{eq:ex}, there can be at most one non-zero term with $p$ greater than $\sqrt{x}$, and so we may use Lemma~\ref{lem:sum primes erdos} and \eqref{eq:sum_pdiv2} to obtain that
\begin{equation}
\EE |E_x| \le C(\log \log x)^{-\frac{1}{6}} + \frac{1}{\sigma_x}\sum_{\substack{ p \in \Primes \setminus \Primes_{x} \\ p \le \sqrt{x}}} \PP( p \mid N_x) + \frac{1}{\sigma_x} \le o(1) + \frac{C}{\sigma_x}\sum_{\substack{ p \in \Primes \setminus \Primes_{x} \\ p \le \sqrt{x}}}\frac{1+\alpha(p)p^{-d}}{p} \to  0,
\end{equation}
where the last expression tends to $0$ by Lemma~\ref{lem:sum on primes} with $g(t)=1/t$. Hence $E_x \overset{d}{\longrightarrow} 0$.  Since $B_x = D_x A_x + E_x$, it follows by Lemma~\ref{lem:sum primes erdos} and the first part of Lemma~\ref{lem:zero exp} that $A_x \overset{d}{\longrightarrow} N(0,1)$ if and only if $B_x \overset{d}{\longrightarrow} N(0,1)$. By the second part of Lemma~\ref{lem:zero exp}, to establish $B_x \overset{d}{\longrightarrow} N(0,1)$ it suffices to show that $\EE B_x^k \to \EE X^k$ for each $k$, where $X \sim N(0,1)$. By Lemma~\ref{lem:dif b c mom} this is equivalent to $\EE C_x^k \to \EE X^k$. Since the random variables $X_p -\alpha(p)/p^{d+1}$ are uniformly bounded as $p$ varies, and since the denominator of $C_x$ tends to infinity by Lemma~\ref{lem:sum primes erdos}, we have $C_x \overset{d}{\longrightarrow} N(0,1)$ by the Lindeberg--Feller theorem (also known as Central Limit Theorem for triangular arrays) \cite[Thm.~4.7]{petrov1995}. Thus, the moments of $C_x$ converge to those of $X$ by Lemma~\ref{lem:mom c bounded} and the last part of Lemma~\ref{lem:zero exp}. \qed 

\section{Polynomially-growing weights}
\begin{lem}\label{lem:dir series poly}
	Fix a function $f\colon \Primes \to \RR_{>0}$ on primes such that $\log f(p) = o(\log p)$, and let
	\begin{equation}
	G_f(s) = \sum_{p}\frac{f(p)}{p^s}, \quad \Re s > 1.
	\end{equation}
	Let $\alpha\colon \NN \to \RR_{\ge 0}$ be a multiplicative function satisfying 
	\begin{equation}
	\begin{split}\label{eq:alpha cond gen}
	&\text{(I) }\alpha(p)= f(p) + O\left( \log^{-2} p\right),\\
	&\text{(II) }\sum_{k \ge 2} \frac{k\alpha(p^k)}{p^k} = O\left( \frac{1}{p\log^2 p} \right),
	\end{split}
	\end{equation}
	and define
	\begin{equation}
	F(s) = \sum_{n \ge 1} \frac{\alpha(n)}{n^s},
	\end{equation}
	the Dirichlet series of $\alpha$. For $\Re s >1$ we have 
	\begin{equation}
	F(s)=\varphi(s)  e^{G_f(s)}
	\end{equation}
	where $\varphi $ is differentiable, bounded and has bounded derivative on $ \Re s \ge 1$.
\end{lem}
\begin{proof}
	By multiplicativity of $\alpha$, we may write 
	\begin{equation}
	F(s)= \prod _{p} \left(  \sum _{k=0}^\infty \frac{\alpha (p^k)}{p^{ks}}  \right)= e^{G_f(s)} \prod _p \left(  1+E_p(s) \right),
	\end{equation}
	where 
	\begin{equation}
	E_p(s)=-1+ \exp \left[ \frac{ -f(p)}{p^s} \right] \cdot \sum _{k=0}^\infty \frac{\alpha (p^k)}{p^{ks}}.
	\end{equation}
	Using \eqref{eq:alpha cond gen} and the triangle inequality, we obtain
	\begin{equation}\label{eq:bound}
	\begin{split}
	\left| \exp \left[ \frac{ f(p)}{p^s} \right] - \sum _{k=0}^\infty \frac{\alpha (p^k)}{p^{ks}} \right| &\le  \left| \exp \left[ \frac{ f(p)}{p^s} \right] - \left( 1+\frac{f(p) }{p^s} \right) \right|+ \left| \frac{\alpha (p)}{p^s} - \frac{f(p)}{p^s} \right|+\sum _{k=2}^\infty \left| \frac{\alpha (p^k)}{p^{ks}} \right|\\
	&\le \frac{C f^2(p) }{p^2}+ \frac{C}{p \log ^2 p}+ \sum _{k=2}^{\infty} \frac{\alpha (p^k)}{p^k} \le \frac{C}{p \log ^2 p}
	\end{split}
	\end{equation}
	for $\Re s \ge 1$. 	Thus, $|E_p(s)| \le C / (p \log ^2 p)$ for $\Re s \ge 1$. 	We turn to bound the derivative of $E_p$. We have 
	\begin{equation}
	E_p '(s)=\log p \cdot \exp \left[ -\frac{f(p)}{p^s} \right] \left( \frac{f(p)}{p^s}  \sum _{k=0}^\infty \frac{\alpha (p^k)}{p^{ks}}-   \sum _{k=1}^\infty \frac{k\alpha (p^k)}{p^{ks}}  \right)
	\end{equation}
	and therefore 
	\begin{equation}
	\left| E_p' (s) \right| \le C \log p \left(   \left| \frac{f(p)}{p^s}-\frac{\alpha (p)}{p^s} \right| +\frac{f(p)}{p} \sum _{k=1}^{\infty} \frac{\alpha (p^k)}{p^k} +\sum _{k=2}^{\infty} \frac{k \alpha (p^k)}{p^k} \right) \le \frac{C}{p \log p}.
	\end{equation}
	Let $p_0>0$ so that for $p\ge p_0$ we have $|E_p(s)|\le 1/2$ for any $s$ with $\Re s\ge 1$. We have that
	\begin{equation}
	F(s)=e^{G_f(s)} \cdot \psi _1 (s) \cdot \psi _2 (s),
	\end{equation}
	where 
	\begin{equation}
	\psi_1 (s):= \prod _{p<p_0} \left( 1+E_p(s) \right), \quad \psi_2 (s):= \prod _{p\ge p_0} \left( 1+E_p(s) \right)=\exp \left[ \sum _{p\ge p_0} \log \left( 1+ E_p (s)\right) \right]
	\end{equation}
	and $\log z$ is the principal branch of the logarithm. The function $\psi _1$ is trivially differentiable, bounded and has bounded derivative on $\Re s \ge 1$. As for $\psi_2$, observe that
	\begin{equation}
	\left| \log (\psi _2 (s)) \right| = | \sum _{p\ge p_0} \log \left( 1+ E_p (s)\right) | \le C\sum _{p\ge p_0} |E_p(s)|\le \sum _{p\ge p_0} \frac{C}{p \log ^2 p } <\infty 
	\end{equation}
	and 
	\begin{equation}
	\left| \frac{d}{ds} \log (\psi _2(s))\right| \le \sum _{p\ge p_0} \left| \frac{E_p'(s) }{1+E_p(s)} \right| \le  C \sum _{p\ge p_0} \left|E_p'(s) \right| \le \sum _{p\ge p_0} \frac{C}{p \log p } <\infty.
	\end{equation}
	It follows that $\psi _2$ is also differentiable, bounded and has bounded derivative on $\Re s \ge 1$. Thus, the same is true for $\varphi :=\psi _1 \cdot \psi _2$, as needed.
\end{proof}

\begin{lem}\label{lem:omega bound}
	Let $\alpha\colon \NN \to \RR_{\ge 0}$ be a multiplicative function satisfying \eqref{eq:polycond2} and
	\begin{equation}
	S(y/h) h \le C \cdot S(y)
	\end{equation}
	for all $y >0$, $h \ge 2$. Then $|\EE \Omega(N_{x}) - \EE\omega(N_x)| \le C$.
\end{lem}
\begin{proof}
	Writing $\omega(N_x)$ as $\sum_{p \le x}  \mathds{1}_{p \mid N_x}$ and $\Omega(N_x)$ as $\sum_{p \le x} \sum_{k \ge 1} \mathds{1}_{p^k \mid N_x}$, we have
	\begin{equation}
	\begin{split}
	0 \le \EE (\Omega(N_{x}) - \omega(N_x))&= \sum_{\substack{k \ge 2\\p} } \PP( p^k \mid N_x) \le \sum_{\substack{k \ge 2\\p} } \sum_{i \ge 0}\alpha(p^{k+i}) \frac{S\big( x/p^{k+i} \big)}{S(x)}\\
	& \le C\sum_{\substack{k \ge 2\\p} } \sum_{i \ge 0} \frac{\alpha(p^{k+i})}{p^{k+i}} \le C\sum_{p} \sum_{ k \ge 2} \frac{k\alpha(p^k)}{p^k} \le C \sum_{p} \frac{1}{p\log^2p} \le C,
	\end{split}
	\end{equation}
	as needed.
\end{proof}
Fix $\gamma>0$ and define the Dirichlet series
\begin{equation}\label{eq:def of G}
G(s)=\sum _{p} \frac{\log ^\gamma  p }{p^s}, \quad \Re s >1.
\end{equation}
\begin{lem}\label{lem:derivatives of G next to 1}
	Fix a non-negative integer $k$. We have 
	\begin{equation}\label{eq:Gderivres}
	G^{(k)}(s)= (-1)^k\frac{\Gamma (\gamma +k ) }{ (s-1)^{\gamma+k }} +B_k + O (|s-1| |s|)
	\end{equation}
	for $ \Re s >1$, where $B_k$ is a real constant depending on $\gamma$ and $k$. Here $(s-1)^{\gamma}=\exp(\gamma \log(s-1))$ is defined using the principal branch of the logarithm.
\end{lem}
\begin{proof}
	We start with the case $k=0$. Let $E(t)= (\sum _{p\le t} \log p)  -t$ be the error term in the prime number theorem. Using integration by parts we obtain, for $\Re s > 1$, that
	\begin{equation}\label{eq:psi 1 psi 2}
	\begin{split}
	\sum _p \frac{\log ^{\gamma } p}{p^s}&= \sum _{n=2}^\infty  \mathds{1} _{\{n \mbox{ is prime}\} } \log n \cdot  \frac{\log ^{\gamma -1} n}{n^s}=-\intop _2 ^{\infty } \bigg( \sum _{p\le t} \log p \bigg) \left(\frac{\log ^{\gamma -1} t}{t^s}\right)'\, \mathrm{d}t \\
	&=\intop _2 ^ \infty \left(t+E(t) \right) \frac{s \log ^{\gamma -1}t-(\gamma -1 ) \log ^{\gamma -2 }t}{t^{s+1}} \, \mathrm{d}t= \psi _1 (s)+\psi _2(s),
	\end{split}
	\end{equation}
	where 
	\begin{equation}
	\psi _1(s):=\intop _2 ^ \infty E(t) \frac{s \log ^{\gamma -1}t-(\gamma -1 ) \log ^{\gamma -2 }t}{t^{s+1}} \, \mathrm{d}t, \quad \psi _2(s):=\intop _2 ^ \infty  \frac{s \log ^{\gamma -1}t-(\gamma -1 ) \log ^{\gamma -2 }t}{t^{s}} \, \mathrm{d}t.
	\end{equation}
	When $\Re s \ge 1$ we may bound $\psi_1$ as follows, using the known result that $E(t)\ll t e^{-c\sqrt{\log t}}$:
\[	\left| \psi _1 (s) \right| \le C|s|\intop _2^\infty \left| e^{-c\sqrt{\log t}}  \frac{\log ^{\gamma -1 }t}{t^s} \right|\, \mathrm{d}t \le C|s|\intop _2^\infty e^{-c\sqrt{\log t}}  \frac{\log ^{\gamma -1}t}{t} \, \mathrm{d}t	\le C|s| \intop _0^\infty e^{-c \sqrt{y}}  y^{\gamma -1 } \, \mathrm{d}y\le C|s|,\]
	and similarly $\psi_1'$ may be bounded by 
	\begin{equation}
	\left| \psi _1 '(s) \right| \le C|s|\intop _2^\infty \left| e^{-c\sqrt{\log t}}  \frac{\log ^{\gamma  }t}{t^s} \right|\, \mathrm{d}t \le C|s| \intop _0^\infty e^{-c \sqrt{y}} y^{\gamma }  \, \mathrm{d}y\le C|s|,
	\end{equation}
	and so
	\begin{equation}\label{eq:psi 1}
	\psi _1(s) = \psi _1(1)+O(|s-1||s|)
	\end{equation}
	in $\Re s \ge 1$. We turn to estimate $\psi _2$. Using integration by parts we obtain
	\begin{multline}
	\intop _2^\infty \frac{\log ^{\gamma -1} t}{t^s}\, \mathrm{d}t=\frac{\log ^{\gamma -1}t}{(1-s)t^{s-1}} \ \bigg| _{t=2}^{t=\infty} - \int_{2}^{\infty} \frac{(\gamma-1) \log^{\gamma-2} t}{(1-s)t^s} \,\mathrm{d}t= \frac{\log ^{\gamma -1}2}{(s-1)2^{s-1}} \\
	 +\frac{1}{s-1}\left(\int_{2}^{\infty} \frac{s \log^{\gamma-1} t}{t^s} \,\mathrm{d}t - \psi_2(s)\right)
	\end{multline}
	for $\Re s > 1$, so that
	\begin{equation}
	\psi_2(s) = \intop _2^\infty \frac{\log ^{\gamma -1} t}{t^s}\, \mathrm{d}t+ \frac{\log^{\gamma-1}2}{2^{s-1}}.
	\end{equation}
	Setting 
	\begin{equation}
	\psi _3 (s):= \intop _1^2 \frac{\log ^{\gamma -1} t}{t^s}\, \mathrm{d}t
	\end{equation}
	we obtain
	\begin{equation}\label{eq:psi 2}
	\begin{split}
	\psi _2 (s)&= \intop _1^ \infty \frac{\log ^{\gamma -1} t}{t^s}\, \mathrm{d}t -\psi _3(s) +\log ^{\gamma -1}2 +O(|s-1|)\\
	&= \intop _1^ \infty \frac{\log ^{\gamma -1} t}{t^s}\, \mathrm{d}t -\psi _3 (1)+\log ^{\gamma -1}2 +O(|s-1|),
	\end{split}
	\end{equation}
	where in the last equality we used that $|\psi _3(s)|,|\psi _3'(s)|\le C$ for $\Re s \ge 1$. In order to compute the integral in the right-hand side of \eqref{eq:psi 2} we perform the change of variables $w=(s-1)\log t$, obtaining
	\begin{equation}\label{eq:change of variables twice2}
	\intop _1^\infty \frac{\log ^{\gamma -1} t}{t^s}\, \mathrm{d}t= \frac{1}{(s-1)^{\gamma} } \intop _0^{(s-1)\infty} w^{\gamma -1} e^{-w}\, \mathrm{d}w.
	\end{equation}
	Letting $C_R$ be the circular contour from $R$ to $R \frac{s-1}{|s-1|}$, we compute that $\lim_{R \to \infty}\int_{C_R} w^{\gamma-1} e^{-w}\mathrm{d}w =0$ when $\Re s > 1$. Thus, we may deform the contour $\{ (s-1)  w: w \ge 0\}$ in \eqref{eq:change of variables twice2} to the positive real line, obtaining 
	\begin{equation}\label{eq:change of variables twice}
	\intop _1^\infty \frac{\log ^{\gamma -1} t}{t^s}\, \mathrm{d}t=\frac{1}{(s-1)^{\gamma} } \intop _0^ \infty w^{\gamma -1} e^{-w}\, \mathrm{d}w=\frac{\Gamma (\gamma  ) }{(s-1)^{\gamma } }.
	\end{equation}	
	Substituting \eqref{eq:change of variables twice} into \eqref{eq:psi 2} and then substituting \eqref{eq:psi 1} and \eqref{eq:psi 2} into \eqref{eq:psi 1 psi 2}, we obtain \eqref{eq:Gderivres} with $k=0$ and 
	\begin{equation}
	B_0= \log ^{\gamma -1}2+\psi _1 (1)-\psi _3 (1).
	\end{equation}
	Next we consider the case $k \ge 1$. We have
	\begin{equation}
	G^{(k)}(s)=(-1)^k \sum _p \frac{ \log ^{\gamma +k } p}{p^s}.
	\end{equation}
	Hence, repeating the above arguments with $\gamma $ replaced by $\gamma +k$ gives the desired result for any $k\ge 1$.
\end{proof}
The following lemma bounds $\Re \left(G (\sigma +it) \right)$ when $t$ is not too large.
\begin{lem}\label{lem:bound for large t}
	There exists $c >0$ with the following property. For $\sigma>1$ sufficiently close to $1$, and for any $t\in \RR $ with $1 \le |t|\le e^{1/(\sigma-1)}$ we have 
	\begin{equation}
	\Re \left(G (\sigma +it) \right) \le (1-c ) \cdot  G(\sigma ).
	\end{equation} 
\end{lem}
To prove Lemma~\ref{lem:bound for large t} we use bounds on primes in short intervals. Information of the form we need is already found in a work of Hoheisel \cite{hoheisel1930}. For ease of presentation, we use a stronger result.
\begin{thm}\cite{heathbrown1988}\label{thm:short intervals}
For any sufficiently large $m$, and any $n\ge m$ with $n-m \ge n^{3/4}$, we have
	\begin{equation}\label{eq:lower bound primes}
	\left| \{ p: p\in [m,n] \text{ prime} \} \right| \ge \frac{n-m}{2 \log n}.
	\end{equation} 
\end{thm}

\begin{proof}[Proof of Lemma~\ref{lem:bound for large t}]
	Let $t \in \RR$ with $1 \le |t| \le e^{1/(\sigma-1)}$. As $ \Re \left(G (\sigma +it) \right)$ is an even function of $t$, we may assume that $t > 0$. Consider the set of integers
	\begin{equation}
	M:=\left\{ n\ge 0 :  \frac{10}{\sigma -1 } \le \frac{2\pi }{t} \left(n +\frac{3}{4}\right) \le \frac{ 20}{\sigma -1 } \right\}.
	\end{equation} 
	Clearly, for $\sigma$ close enough to $1$ we have $|M|\ge t /(\sigma -1 )$. For any $n\in M$ we define
	\begin{equation}
	A_n:=\left\{ p : t\log p\in \left[ 2\pi n+\frac{\pi }{2}, 2\pi n +\frac{3\pi}{2}   \right]\right\}=\left\{p: p\in \left[ e^{-\frac{\pi }{t}}x_n, x_n \right]\right\},
	\end{equation}
	where
	\begin{equation}
	x_n:=\exp \left(\frac{2\pi }{t} \left(n +\frac{3}{4}\right)\right).
	\end{equation}
	For any $n \in M$ and $p\in A_n$ we have that $\cos (t \log p ) \le 0$ and therefore 
	\begin{equation}\label{eq:Re g upper}
	\Re \left(G(\sigma +it) \right) =\sum _p \frac{\log ^{\gamma }p}{p^\sigma } \cos (t \log p ) \le G(\sigma )-\sum _{n\in M} \sum _{p \in A_n} \frac{\log ^{\gamma }p}{p^{\sigma }}.
	\end{equation}
	Since $\log ^{\gamma }p/ p^{\sigma }$ is decreasing for sufficiently large $p$ (independently of $\sigma \ge 1$) and as $\min _{n\in M} \min A_n \to \infty $ as $\sigma \to 1^+$ (uniformly in $t \ge 1$) we get that, when $\sigma$ is close enough to $1$,
	\begin{equation}\label{eq:suman}
	\sum _{p\in A_n} \frac{\log ^{\gamma }p }{p^\sigma } \ge |A_n| \frac{\log ^{\gamma }x_n}{x_n^{\sigma }} \ge \frac{\left(1-e^{-\frac{\pi }{t}}\right)x_n }{2 \log x_n} \frac{\log ^{\gamma }x_n}{x_n^{\sigma }} \ge c \frac{\log ^{\gamma -1}x_n}{t \cdot  x_n^{\sigma -1 }} \ge \frac{c}{t} \left(\frac{1}{\sigma -1 } \right)^{\gamma -1}
	\end{equation}
	where in the second inequality we used \eqref{eq:lower bound primes}. Indeed the conditions of Theorem~\ref{thm:short intervals} hold for $[e^{-\pi/t}x_n,x_n]$ as for any $n\in M$ we have 
	\begin{equation}
	x_n \ge \exp \left( \frac{10 }{\sigma -1 }\right)\ge t^{10}	\end{equation}
	and so
	\begin{equation}
	x_n -e^{-\frac{\pi }{t}}x_n = x_n \left(1- e^{-\frac{\pi }{t}}  \right) \ge \frac{c}{t}x_n \ge c x_n^{0.9 }.
	\end{equation}
	Summing \eqref{eq:suman} over $n \in M$ we get
	\begin{equation}\label{eq:sum an n}
	\sum _{n\in M} \sum _{p\in A_n} \frac{\log ^{\gamma }p}{p^{\sigma }} \ge \frac{c|M|}{t} \left( \frac{1}{\sigma -1 }\right)^{\gamma -1} \ge c \left(\frac{1}{\sigma -1 } \right) ^{ \gamma } \ge c \cdot G(\sigma ),
	\end{equation}
	where in the last inequality we used Lemma~\ref{lem:derivatives of G next to 1} with $k=0$. From \eqref{eq:Re g upper} and \eqref{eq:sum an n} we obtain the desired bound.
\end{proof}

It turns out that when $|t|\ge e^{1/(\sigma -1)}$, the result of Lemma~\ref{lem:bound for large t} does not hold and $\Re \left(G (\sigma +it) \right)$ might be as large as $G(\sigma )$. We shall show that the reals  $t\in \RR $ for which $\Re \left(G (\sigma +it) \right)$ is as large as $G(\sigma )$ are quite rare.
More precisely, we will show in the following lemma that if $t_1$, $t_2$ are such that $\Re \left(G (\sigma +it) \right)$ is large then the same holds for $t_1-t_2$, and therefore by Lemma~\ref{lem:bound for large t} $t_1$ and $t_2$ must be far away from each other.
\begin{lem}\label{lem:t_1 t_2}
	Let $\{a_n\}_{n \ge 1}$ be a sequence of non-negative reals with $\sum _{n=1}^{\infty } a_n<\infty$. Consider the function 
	\begin{equation}
	f(t)=\sum _{n=1 }^{\infty} a_n \cos (t \log n), \quad t\in \RR.
	\end{equation}
	For any $0<\varepsilon <1$ and any $t_1,t_2\in \RR $ with 
	\begin{equation}
	f(t_1)\ge (1-\varepsilon )f(0), \quad f(t_2)\ge (1-\varepsilon )f(0)
	\end{equation}
	we have 
	\begin{equation}
	f(t_1-t_2)\ge (1-8 \sqrt{\varepsilon })f(0).
	\end{equation}
\end{lem}

\begin{proof}
	Let $0<\varepsilon <1$ and define 
	\begin{equation}
	A_t :=\{n\ge 1 : \cos (t \log n ) >1-\sqrt{\varepsilon }\}
	\end{equation}
	for $t \in \{t_1,t_2\}$. By the assumption on $t_1$ we have 
	\begin{equation}
	(1-\varepsilon )f(0) \le  f(t_1)= \sum _{n=1}^{\infty } a_n \cos (t_1 \log n ) \le \sum _{n\in A_{t_1}} a_n +(1-\sqrt{\varepsilon })\sum _{n\notin A_{t_1}} a_n = f(0) - \sqrt{\varepsilon} \sum_{n \notin A_{t_1}} a_n,
	\end{equation}
	so that
	\begin{equation}\label{eq:At1ineq}
	\sum_{n \notin A_{t_1}}a_n \le \sqrt{\varepsilon} f(0).
	\end{equation}
	The same argument shows that \eqref{eq:At1ineq} holds for $t_2$ in place of $t_1$ as well. Therefore, by the union bound,
	\begin{equation}\label{eq:bounds on intersection}
	\sum _{n\notin A_{t_1} \cap A_{t_2}} a_n \le 2\sqrt{\varepsilon } f(0) \quad \mbox{and} \quad \sum _{n\in A_{t_1} \cap A_{t_2}} a_n \ge (1-2\sqrt{\varepsilon }) f(0). 
	\end{equation}	
	Now, for any $n\in A_{t_1} \cap A_{t_2}$ and $i=1,2$ we have
	\begin{equation}
	|\sin (t_i \log n )| =\sqrt{1-\cos ^2 (t_i \log n )} \le \sqrt{1-(1-\sqrt{\varepsilon })^2} \le  \sqrt{1-(1-2 \sqrt{\varepsilon} )}= \sqrt{2} \varepsilon ^{\frac{1}{4}}.
	\end{equation} 
	Thus, for any $n \in A_{t_1} \cap A_{t_2}$,
	\begin{multline}\label{eq:cos of sum}
	\cos \left((t_1 -t_2 ) \log n \right)= \cos (t_1 \log n ) \cos (t_2 \log n ) +\sin (t_1 \log n )\sin ( t_2 \log n )\\
	\ge (1- \sqrt{\varepsilon })^2-2\sqrt{\varepsilon } \ge 1-4 \sqrt{\varepsilon }.
	\end{multline}
	From \eqref{eq:bounds on intersection} and \eqref{eq:cos of sum} we obtain that
	\begin{equation}
	\begin{split}
	f(t_1 -t_2)=\sum _{n=1}^{\infty } a_n \cos \left( (t_1-t_2 ) \log n \right) &\ge (1-4 \sqrt{\varepsilon }) \sum _{n\in A_{t_1} \cap A_{t_2}} a_n-\sum _{n \notin A_{t_1} \cap A_{t_2}} a_n\\
	&\ge (1-4 \sqrt{\varepsilon })(1-2 \sqrt{\varepsilon }) f(0)-2 \sqrt{\varepsilon } f(0) \ge (1-8 \sqrt{\varepsilon } ) f(0),
	\end{split}
	\end{equation}
	as needed.
\end{proof}
Fix $\PolyConst >0$. The function $G'(s)$ is monotone-increasing for real $s>1$, with $\lim_{s \to \infty} G'(s) = 0$ and $\lim_{s\to 1^+} G'(s) = -\infty$. For $x>1$, we let $\sigma = \sigma _x$ be the unique real solution to
\begin{equation}\label{eq:def of sigma }
\PolyConst \cdot G'(\sigma )=-\log x.
\end{equation}
The point $\sigma$ plays the role of the saddle point in the proof of Theorem~\ref{thm:PolyPartFunc}.
The following is a corollary of Lemma~\ref{lem:derivatives of G next to 1}.
\begin{cor}\label{cor:derivatives at he saddle point}
	As $x\to \infty$ we have
	\begin{equation}\label{eq:sigma est}
	\sigma -1 =\left( 1+O\left( \frac{1}{\log x } \right) \right) \left( \PolyConst     \Gamma (\gamma +1) \right)^{\frac{1}{\gamma +1}}\left( \log x  \right)^{-\frac{1}{\gamma +1}},
	\end{equation}
	and moreover 
	\begin{equation}
	\begin{split}
	G(\sigma )&= \frac{ \Gamma (\gamma )}{(\PolyConst\Gamma (\gamma +1))^{\frac{ \gamma }{\gamma +1}}}\left( \log x \right)^{\frac{\gamma }{\gamma +1}}+B_0+O\left( (\log x)^{-\frac{1 }{\gamma +1}} \right),\\
	G''(\sigma )&= \frac{\Gamma (\gamma+2 )}{(K\Gamma (\gamma +1))^{\frac{\gamma +2 }{\gamma +1}}}\left( \log x \right)^{\frac{\gamma +2 }{\gamma +1}} +O\left( (\log x )^{\frac{1}{\gamma +1}} \right), \\
	G'''(\sigma )&\sim -\frac{ \Gamma (\gamma +3)}{(K \Gamma (\gamma +1))^{\frac{\gamma +3 }{\gamma +1}}}\left( \log x \right)^{\frac{\gamma +3 }{\gamma +1}}.
	\end{split}
	\end{equation}
\end{cor}

\begin{proof}
	Since $\lim_{s\to 1^+} G'(s)=-\infty$, it follows that $\sigma_x = O(1)$ for $x \ge 2$. Using Lemma~\ref{lem:derivatives of G next to 1}, \eqref{eq:def of sigma } becomes 
	\begin{equation}
	\PolyConst \frac{\Gamma (\gamma +1)}{(\sigma -1)^{\gamma +1}}= \log x +O(1), 
	\end{equation}
	from which we deduce \eqref{eq:sigma est}. Applying the estimates for $G^{(k)}(s)$ in Lemma~\ref{lem:derivatives of G next to 1} for $s=\sigma$ and $k=0,2,3$, and using \eqref{eq:sigma est}, we obtain the stated estimates for $G(\sigma)$, $G''(\sigma )$, $G'''(\sigma )$.
\end{proof}
\subsection{Proof of Theorem~\ref{thm:PolyPartFunc}}
If we replace $x$ with $\lfloor x \rfloor + \frac{1}{2}$ in \eqref{eq:sum of alpha poly}, then the left-hand side does not change, while the function in the right-hand side is changed by a factor of $1+O(1/x)$, which can be absorbed in the relative error term. Thus, in proving Theorem~\ref{thm:PolyPartFunc} we may assume without loss of generality that $x$ is of the form $m+\frac{1}{2}$ for some positive integer $m$ (i.e.~half-integer). We denote by $F(s)$ the Dirichlet series of $\alpha $, which by Lemma~\ref{lem:dir series poly} is of the form $F(s)=\varphi(s)  e^{\PolyConst \cdot G(s)}$ where $\varphi $ is differentiable, bounded and has bounded derivative on $ \Re s \ge 1$. 
By an effective version of Perron's formula \cite[Thm.~II.2.3]{tenenbaum2015} we have  
\begin{equation}
\sum _{n \le x} \alpha (n)=\frac{1}{2\pi i} \intop _{\sigma -iT}^{\sigma+iT} F(s)x^s \, \frac{\mathrm{d}s}{s} + O\left( x^{\sigma}\sum_{n \ge 1} \frac{\alpha(n)}{n^{\sigma} (1+T | \log(x/n) | )} \right)
\end{equation}
for any $T \ge 1$, where $\sigma=\sigma_x$ is defined in \eqref{eq:def of sigma }. We choose $T=x^2$, obtaining
\begin{equation}\label{eq:perron with specific T}
\sum _{n \le x} \alpha (n)=\frac{1}{2\pi }\intop _{-x^2}^{x^2} \varphi (\sigma +it) e^{\PolyConst \cdot G(\sigma +it)}\frac{x^{\sigma +it}}{\sigma +it}\, \mathrm{d}t + O(E(x)),
\end{equation}
where
\begin{equation}\label{eq:Ex est}
E(x) = x^{\sigma}\sum_{n \ge 1} \frac{\alpha(n)}{n^{\sigma} (1+x^2 | \log(x/n) | )} \le C \frac{x^{\sigma }}{x} \sum _{n \ge 1} \frac{\alpha (n)}{n^{\sigma }} = C \frac{x^{\sigma } F(\sigma ) }{x} \le  C \frac{x^{\sigma } e^{\PolyConst \cdot G(\sigma )}  }{x}.
\end{equation}
(We have used the fact that $x$ is an half-integer and so $\left| \log (x/n) \right| \ge  C/x$.) Set 
\begin{equation}
t_x:=\left( \log x \right)^{\frac{\delta}{6(\gamma+1)} -\frac{\gamma +2}{2(\gamma +1)}}
\end{equation}
for some sufficiently small $\delta>0$. We decompose the integral in the right-hand side of \eqref{eq:perron with specific T} into three parts, to be estimated in the following ways:
\begin{equation}
\begin{split}
|t|\le t_x: \quad &\mbox{estimated using Corollary~\ref{cor:derivatives at he saddle point},} \\
t_x \le |t|\le 1: \quad  &\mbox{bounded using Lemma~\ref{lem:derivatives of G next to 1},} \\
1 \le |t| \le x^2: \quad  &\mbox{bounded using Lemmas~\ref{lem:bound for large t} and \ref{lem:t_1 t_2}.}
\end{split}
\end{equation}
We denote by $I_1, I_2,I_3$ the integrals over these respective domains. 
We begin by computing the asymptotics for $I_1$, which gives the main term. When $|t|\le t_x$ we have, by Corollary~\ref{cor:derivatives at he saddle point},
\begin{equation}\label{eq:varphiapprox}
\frac{\varphi(\sigma+it)}{\sigma+it} = \varphi(1) \left( 1 + O\left( t_x + \sigma-1 \right) \right) = \varphi(1) \left( 1 + O\left(  \log^{-\frac{1}{\gamma+1}} x\right) \right)
\end{equation}
since $\varphi$ has bounded derivative. A second-order Taylor approximation of $G(\sigma+it)$ around $t=0$ gives
\begin{equation}\label{eq:secondtaylor}
\begin{split}
\PolyConst \cdot  G(\sigma +it)&=\PolyConst \cdot G(\sigma ) +it \PolyConst \cdot G'(\sigma ) -\frac{t^2}{2}\PolyConst \cdot  G''(\sigma) + O\left( |t_x^3 G'''(\sigma )|\right) \\
&=  \PolyConst \cdot G(\sigma ) -it \log x -\frac{t^2}{2}\PolyConst \cdot  G''(\sigma ) +O\left(\log ^{ -\frac{\gamma-\delta}{2(\gamma+1)}} x\right)
\end{split}
\end{equation}
for $|t|\le t_x$, where in the first equality we used the fact that $|G'''(\sigma +it)| \le |G'''(\sigma )|$  and in the second equality we used Corollary~\ref{cor:derivatives at he saddle point} and the definition of $\sigma $ in \eqref{eq:def of sigma }. From \eqref{eq:varphiapprox} and \eqref{eq:secondtaylor}, we get 
\begin{equation}
\varphi (\sigma +it) e^{\PolyConst \cdot  G(\sigma +it)}\frac{x^{\sigma +it}}{\sigma +it}= \varphi (1) x^{\sigma } e^{\PolyConst \cdot G(\sigma ) -\frac{t^2}{2}\PolyConst \cdot  G''(\sigma )} \left( 1+O\left(\log ^{-z} x\right) \right)
\end{equation}
where $z:=\min\{\gamma-\delta,2\}/(2(\gamma+1))$. We thus have
\begin{equation}
I_1= \left( 1+O\left( \log ^ {-z} x \right)\right)\frac{\varphi (1) x^{\sigma } e^{\PolyConst \cdot  G(\sigma )}}{2 \pi } \intop _{- t_x} ^{t_x} e^{-\frac{\PolyConst \cdot  G''(\sigma ) }{2} t^2}\, \mathrm{d}t= \left( 1+O\left( \log ^{-z} x \right)\right)\frac{\varphi (1) x^{\sigma } e^{\PolyConst \cdot  G(\sigma )}}{\sqrt{2 \pi \PolyConst \cdot  G''(\sigma )} },
\end{equation}
which by Corollary~\ref{cor:derivatives at he saddle point} can be simplified to
\begin{equation}\label{eq:I1 asymptotics}
I_1 = \left( 1+O\left(\log ^{-z} x \right)\right) A_{\alpha}  x ( \log x )^{-\frac{\gamma +2 }{2( \gamma +1)}}   \exp \left[B (\log x )^{\frac{\gamma }{\gamma +1}}\right],
\end{equation}
where $B$ is defined in \eqref{eq:B def}, and $A_{\alpha}$ is defined in \eqref{eq:poly consts}.
Next we bound $I_2$. Using Lemma~\ref{lem:derivatives of G next to 1} with $s=\sigma+it$ where $t_x \le |t|\le 1$, we get
\begin{equation}
\Re \left( G(\sigma +it) \right) \le  \left|G (\sigma +it)\right|\le \frac{\Gamma (\gamma )}{\left|\sigma -1 +it\right|^{\gamma }} +C \le \frac{\Gamma (\gamma )}{\left|\sigma -1 +it_x\right|^{\gamma }} +C \le \left|G (\sigma +it_x ) \right|+C,
\end{equation}
and so a second-order Taylor approximation shows that
\begin{equation}
\begin{split}
\Re \left( G(\sigma +it) \right)&\le \left| G(\sigma ) +it_xG'(\sigma )-\frac{t_x^2  }{2}G''(\sigma ) \right| +C\\
&= \sqrt{G(\sigma )^2-t_x^2 \left(G(\sigma )G''(\sigma )-G'(\sigma )^2\right)+t_x^4G''(\sigma )^2/4}+C\\
&\le \sqrt{G(\sigma ) ^2- c(\log x )^{\frac{3 \gamma +\delta }{3 (\gamma +1 )}}  } +C \le G (\sigma )- c (\log x )^{\frac{\delta }{3 (\gamma +1)}},
\end{split}
\end{equation}
where the second inequality holds for sufficiently small $\delta $ and follows from Corollary~\ref{cor:derivatives at he saddle point}. Thus 
\begin{equation}\label{eq:I2 bound}
|I_2| \le C x^{\sigma } e^{\PolyConst \cdot G(\sigma ) - c (\log x)^{\frac{\delta }{3(\gamma+1)}}  } \le \frac{C}{\log x } |I_1|.
\end{equation}
We now show that the contribution from $I_3$ is negligible as well. Fix $\varepsilon \in (0,1)$ such that $8 \sqrt{\varepsilon}$ is strictly less than the constant $c$ from Lemma~\ref{lem:bound for large t}. Consider the set 
\begin{equation}
S:=\{t >0 : \Re \left( G(\sigma +it) \right)>(1-\varepsilon )G(\sigma ) \}.
\end{equation}
We have, by definition of $S$,
\begin{equation}
\begin{split}
\intop _{\left[1, x^2  \right]\setminus S} \left| \varphi (\sigma+it) e^{\PolyConst \cdot G(\sigma +it)}\frac{x^{\sigma +it}}{\sigma +it}\right| \, \mathrm{d}t  &\le C x^{\sigma }e^{(1-\varepsilon )\PolyConst \cdot G(\sigma )} \intop _{\left[1, x^2  \right]\setminus S} \frac{\mathrm{d}t}{t+1} \\
&\le C x^{\sigma } e^{(1-\varepsilon )\PolyConst \cdot G(\sigma )} \log x \le \frac{C}{\log x} |I_1|,
\end{split}
\end{equation}
where in the last inequality we used Corollary~\ref{cor:derivatives at he saddle point}. We now study the integral over $S$. Applying Lemma~\ref{lem:t_1 t_2} with the sequence 
\begin{equation}
a_n:= \mathds{1}_{\{ n \text{ is prime}\} } \frac{\log ^{\gamma }n}{n^{\sigma }}
\end{equation}
we find that for any $t_1, t_2 \in S$ we have that $\Re  \left( G(\sigma +i(t_1-t_2)) \right) \ge (1-8 \sqrt{\varepsilon }   )G(\sigma )$ and therefore, by Lemma~\ref{lem:bound for large t} and the choice of $\varepsilon $, either $|t_1 -t_2|\le 1 $ or $|t_1 -t_2 |\ge  e^{1/(\sigma -1)}$. It follows that 
\begin{equation}
S\subseteq \bigcup _{j=0}^ \infty [a_j,b_j]
\end{equation}
for some $a_{j+1} > b_j > a_j \ge 0$ with 
\begin{equation}
b_j-a_j \le 1, \quad a_j \ge j \cdot e^{\frac{1 }{\sigma -1 }} \quad \text{ and }\quad  a_0=0.
\end{equation}
Thus, for sufficiently large $x$,
\begin{equation}
\begin{split}
\intop _{\left[1, x^2  \right]\cap S} \left| \varphi (\sigma +it)  e^{\PolyConst \cdot  G(\sigma +it)}\frac{x^{\sigma +it}}{\sigma +it}\,\mathrm{d}t \right| &\le C x^{\sigma } e^{\PolyConst \cdot  G(\sigma )} \sum _{1 \le j \le x^2}  \intop _{a_j} ^{b_j} \frac{\mathrm{d}t}{t} \le C x^{\sigma } e^{\PolyConst \cdot  G(\sigma )} \sum _{1 \le j \le x^2} \frac{b_j -a_j}{a_j} \\ 
&\le C x^\sigma e^{\PolyConst \cdot  G(\sigma )} e^{-\frac{1 }{\sigma -1 }} \sum _{1 \le j \le x^2 } \frac{1}{j} \le C x^\sigma e^{\PolyConst \cdot  G(\sigma )} e^{-\frac{1 }{\sigma -1 }} \log x \\
&\le \frac{C}{\log x} |I_1|.
\end{split}
\end{equation}
Combining the estimates for the integrals over $\left[1, x^2  \right]\setminus S$ and $\left[1, x^2  \right]\cap S$, we obtain
\begin{equation}\label{eq:I3 bound}
|I_3|\le \frac{C}{\log x}|I_1|.
\end{equation}
We conclude the proof by plugging the estimates \eqref{eq:Ex est}, \eqref{eq:I1 asymptotics}, \eqref{eq:I2 bound} and \eqref{eq:I3 bound} in \eqref{eq:perron with specific T}.	\qed
\subsection{Proof of Theorem~\ref{thm:integer_poly}}
\subsubsection{Auxiliary results}
An important step in the proof is understanding the asymptotic behavior of $\PP(p \mid N_{x})$. We shall see that $\PP(p \mid N_{x})\approx \alpha (p) S(x/p)/S(x)$, and so begin by studying the ratio $S(x/h)/S(x)$. Observe that 
\begin{equation}\label{eq:ratio poly}
S(x/h)h \le C \cdot S(x)
\end{equation}
for $x \ge 1$, $h \ge 1$ by Theorem~\ref{thm:PolyPartFunc}. 
\begin{lem}\label{lem:s poly}
	Let $\alpha\colon \NN\to \RR_{\ge 0}$ be a multiplicative function satisfying \eqref{eq:polycond1}--\eqref{eq:polycond2} and suppose that $x$ is sufficiently large. Let $2\le h \le x$. When $\log h \le (\log x)^{(\gamma+4)/(4\gamma+4)}$,  we have that 
	\begin{equation}
	\frac{S(x/h)}{S(x)} = \frac{1}{h}\exp \bigg[ -\frac{B \gamma }{\gamma +1 } \frac{\log h }{ (\log x)^{\frac{1}{\gamma +1}}} + O\bigg( \frac{1}{(\log x)^c}\bigg) \bigg].
	\end{equation}
	When $\log h \ge (\log x)^{(\gamma+4)/(4\gamma+4)}$,  we have that
	\begin{equation}
	\frac{S(x/h)}{S(x)} \le  \frac{1}{h}e^{- (\log x)^{c}}.
	\end{equation}
\end{lem}
\begin{proof}
	Let $h_x:= \exp((\log x)^{(\gamma+4)/(4\gamma+4)})$. Suppose that $h \le h_x$. By a first-order Taylor approximation, we get
	\begin{equation}\label{eq:second taylor}
	(\log (x/h))^{\frac{\gamma}{\gamma+1}} = (\log x)^{\frac{\gamma}{\gamma+1}}  \left( 1-\frac{\log h }{\log x }  \right) ^{\frac{\gamma }{\gamma +1}}= (\log x)^{\frac{\gamma }{\gamma +1}} - \frac{\gamma }{\gamma +1} \frac{\log h }{ (\log x)^{\frac{1}{\gamma +1}} } +O \left( \frac{1}{(\log x) ^c}\right).
	\end{equation}
	Thus, by Theorem~\ref{thm:PolyPartFunc} applied with $x$ and $x/h$, we obtain the first part of the lemma.
	We turn to prove the second part of the lemma. Using the first part of the lemma and \eqref{eq:ratio poly} we get that when $h\ge h_x$ 
	\begin{equation}
	\frac{S(x/h)}{S(x)} = \frac{S(x/h)}{S(x/h_x)}\frac{S(x/h_x)}{S(x)}\le \frac{Ch_x}{h}\frac{1}{h_x} \exp \left[  -c\frac{\log h_x }{ (\log x)^{\frac{1}{\gamma +1}}} \right] \le \frac{1}{h} e^{-(\log x)^c},
	\end{equation}
	as needed.
\end{proof}
\begin{lem}\label{lem:prop p poly}
	Let $\alpha\colon \NN\to \RR_{\ge 0}$ be a multiplicative function satisfying $\log \alpha(p) = o(\log p)$, \eqref{eq:polycond2} and \eqref{eq:ratio poly}. For sufficiently large $x$ and any prime $p\le x$ we have that 
	\begin{equation}
	\PP \left( p \mid N_x \right)= \frac{\alpha(p) S(x/p)}{S(x)}+O \bigg( \frac{1}{p \log ^2 p} \bigg).
	\end{equation} 
\end{lem}
\begin{proof}
	By \eqref{eq:ratio poly} we have for any $y>1$
	\begin{equation}\label{eq:breaking to maximal prime power poly}
	\sum _{\substack{n\le y \\ p \mid n}} \alpha (n)\le \sum _{k \ge 1} \sum _{\substack{n \le y \\ p^k \mid \mid n}} \alpha (n) \le \sum _{k \ge 1} \alpha (p^k) S\Big(\frac{y}{p^k} \Big) \le  C \cdot S(y) \sum _{k \ge 1} \frac{\alpha(p^k)}{p^k} \le \frac{C\cdot S(y)}{\sqrt{p} }.
	\end{equation}
Hence
	\begin{equation}
	\begin{split}
	\PP \left(p \mid \mid N_{x} \right) = \frac{1}{S(x)}\sum _{\substack{n\le x\\ p\mid \mid n}} \alpha (n)=\frac{\alpha (p)}{S(x)} \sum _{\substack{m \le x/p \\ p\nmid m}} \alpha (m)&=\frac{\alpha (p) S(x/p)}{S(x)}+ O \left( \frac{ \alpha (p) S(x/p)}{\sqrt{p} S(x) } \right)\\
	&= \frac{\alpha(p)S(x/p)}{S(x)}+O\left( \frac{1}{p \log ^ 2 p}\right),
	\end{split}
	\end{equation}
	where in the last inequality we use \eqref{eq:ratio poly} again. Using the same arguments as in \eqref{eq:breaking to maximal prime power poly}, we have that 
	\begin{equation}
	\PP (p^2 \mid N_{x}) \le C \sum _{k=2}^{\infty} \frac{\alpha (p^k)}{p^k} \le \frac{C}{p \log ^ 2 p}.
	\end{equation}
	As $\PP (p \mid N_{x}) = \PP (p \mid \mid N_{x}) + \PP (p^2 \mid N_{x})$, the proof is concluded.
\end{proof}
\subsubsection{Conclusion of proof}
We begin with the first part of the theorem. We abbreviate $P_1(N_{x})$ as $P_1$. Fix $0<a<b<\infty$. It suffices to show that
\begin{equation}\label{eq:liminf typ}
\liminf_{x \to \infty} \PP \left( a\le \frac{\log P_1}{(\log x)^{\frac{1}{\gamma+1}} }\le b \right) \ge \PP \left( a \le Y \le b\right)
\end{equation}
where $Y$ has $\gammadist( \gamma+1,(K\Gamma(\gamma+1))^{1/(\gamma+1)})$ distribution. 

For any prime $p$ such that $a (\log x)^{\frac{1}{\gamma +1}} \le \log p \le b (\log x)^{\frac{1}{\gamma +1}}$, we have, by Lemmas~\ref{lem:s poly} and \ref{lem:prop p poly},
\begin{equation}\label{eq:typlopineq}
\begin{split}
\PP ( P_1 = p ) &=\sum _{\substack{ n\le x \\ p\mid n }} \frac{\alpha (n)}{S(x)}  \frac{\nu _p (n) \log p }{\log n} \ge    \frac{\log p }{\log x } \PP (p \mid N_x) \ge  \frac{\log p }{\log x } \left( \frac{\PolyConst \log^{\gamma} p S(x/p)}{S(x)}+O \bigg( \frac{1}{p \log ^2 p} \bigg) \right)\\
& \ge \left( 1+O\left( \frac{1}{(\log x)^c} \right) \right) \frac{ \PolyConst \log ^{\gamma +1} p }{p \log x } \exp \left[ -\frac{B \gamma }{\gamma +1 } \frac{\log p }{(\log x)^{\frac{1}{\gamma +1}}} \right],
\end{split}
\end{equation}
where the error term $1/(p\log^2p)$ is absorbed in the last error term. Thus,
\begin{equation}
\PP \left( a \le \frac{\log P_1}{(\log x)^{1/(\gamma+1)} } \le b \right)\ge \left( 1+O\left( \frac{1}{(\log x)^c} \right) \right)\sum _{n_x^a \le p \le n_x ^b} \frac{ \PolyConst \log ^{\gamma +1} p }{p \log x } \exp \left[ -\frac{B \gamma }{\gamma +1 } \frac{\log p }{(\log x)^{\frac{1}{\gamma +1}}} \right]
\end{equation}
where $n_x := \exp( (\log x)^{1/(\gamma+1)} )$. By Lemma~\ref{lem:sum on primes} with $\alpha \equiv 1$, the interval $[n_x^a, n_x^b]$ and 
\begin{equation}
g(t) =  \frac{ \PolyConst \log ^{\gamma +1} t }{t \log x } \exp \bigg[ -\frac{B \gamma }{\gamma +1 } \frac{\log t }{(\log x)^{\frac{1}{\gamma +1}}} \bigg]
\end{equation}
we have
\begin{multline}\label{eq:p typ to int}
\PP \left( a \le \frac{\log P_1}{(\log x)^{1/(\gamma+1)}} \le b \right) \ge  \left( 1+O\left( \frac{1}{(\log x)^c} \right)  \right) \\
\left( \intop _{n_x^a } ^{n_x^b} \frac{\PolyConst \log^{\gamma } t }{t\log x} \exp \left[ -\frac{B \gamma }{\gamma +1 } \frac{\log t }{ (\log x)^{\frac{1}{\gamma +1}} } \right]\,\mathrm{d}t + O(\log^{-\frac{1}{\gamma+1}}x) \right).
\end{multline}
The change of variables $t=n_x^z$ in the last integral shows that it equals $\PP(a \le Y \le b)$. Taking $x$ to infinity in \eqref{eq:p typ to int} we obtain \eqref{eq:liminf typ}, as needed.

We turn to the second part of the theorem. By Lemma~\ref{lem:prop p poly} and \eqref{eq:polycond1} we have
\begin{equation}
\EE  \omega(N_{x}) = \sum_{p \le x} \PP( p \mid N_{x})=	\sum_{p \le x} \frac{\PolyConst \log^{\gamma} p S(x/p)}{S(x)}+O \bigg( \sum_p \frac{1}{p \log ^2 p} \bigg).
\end{equation}
The error term is bounded by a constant. In order to estimate the sum, we split it into three sums $S_{1}$, $S_{2}$ and $S_{3}$, over the respective ranges $p < \exp((\log x)^{\delta_1})$, $\exp((\log x)^{\delta_1}) \le p \le \exp((\log x)^{\delta_2})$ and $\exp((\log x)^{\delta_2}) < p \le x$, where $\delta_1 = 1/(2(\gamma+1))$ and $\delta_2 = (\gamma+4)/(4\gamma+4)$. We bound $S_{1}$ using \eqref{eq:ratio poly}:
\begin{equation}
S_{1}  \le C \sum_{p \le \exp((\log x)^{\delta_1}) } \frac{\log^{\gamma} p }{p} \le C \log^{ \delta_1 \gamma} x,
\end{equation}
where in the last inequality we used Lemma~\ref{lem:sum on primes} with $\alpha \equiv 1$ and $g(t)= \log^{\gamma} t /t$. We bound $S_{3}$ using the second part of Lemma~\ref{lem:s poly}, which gives 
\begin{equation}
S_{3} \le e^{- (\log x)^{c}}\sum_{\exp((\log x)^{\delta_2})  <p \le x} \frac{\PolyConst \log^{\gamma} p}{p} \le e^{- (\log x)^{c}}.
\end{equation}
We now estimate $S_{2}$. By Lemma~\ref{lem:s poly},
\begin{equation}\label{eq:s12 1}
S_{2} = \left(1 + O\left(\frac{1}{(\log x)^c}\right) \right) \sum_{(\log x)^{\delta_1} \le \log p \le (\log x)^{\delta_2}} \frac{\PolyConst \log^{\gamma} p}{p}  \exp \left[ -\frac{B \gamma }{\gamma +1 } \frac{\log p }{ (\log x)^{\frac{1}{\gamma +1}}} \right].
\end{equation}
From \eqref{eq:s12 1} and Lemma~\ref{lem:sum on primes} with $\alpha \equiv 1$ and
\begin{equation}
g(t) =  \frac{ \PolyConst \log ^{\gamma} t }{t } \exp \left[ -\frac{B \gamma }{\gamma +1 } \frac{\log t }{ (\log x)^{\frac{1}{\gamma +1}}} \right]
\end{equation}
we obtain 
\begin{equation}
S_{2} =  \left(1+O\left(\frac{1}{(\log x)^c} \right)\right)\left( \int_{\exp((\log x)^{\delta_1}) }^{\exp((\log x)^{\delta_2})} \frac{g(t)}{\log t} \mathrm{d}t + O\left( \log^{\max\{0,\frac{\gamma-1}{\gamma+1}\}} x\right) \right).
\end{equation}
The change of variables $t=n_x^z$ in the last integral shows that it equals 
\begin{multline}
(\log x)^{\frac{\gamma}{\gamma+1}}\intop_{(\log x)^{\delta_1-\frac{1}{\gamma+1}}}^{(\log x)^{\delta_2-\frac{1}{\gamma+1}}} \PolyConst z^{\gamma-1} \exp\left( -\frac{B \gamma z}{\gamma+1} \right) \, \mathrm{d}z \\
= (\log x)^{\frac{\gamma}{\gamma+1}} \PolyConst \Gamma(\gamma) \left( \frac{\gamma+1}{B\gamma}\right)^{\gamma} \left(1 + O\left( (\log x)^{\gamma(\delta_1 - \frac{1}{\gamma+1})}\right)\right).
\end{multline}
Since all the accumulated error terms are of order smaller than $(\log x)^{\gamma/(\gamma+1)}$, the expectation of $\omega$ is estimated. The expectation of $\Omega$ behaves the same by Lemma~\ref{lem:omega bound}.  \qed

\bibliographystyle{abbrv}
\bibliography{references}

\begin{thebibliography}{10}

\bibitem{alladi1982}
K.~Alladi.
\newblock Additive functions and special sets of integers.
\newblock In {\em Number theory ({M}ysore, 1981)}, volume 938 of {\em Lecture
  Notes in Math.}, pages 1--49. Springer, Berlin-New York, 1982.

\bibitem{alladi1984}
K.~Alladi.
\newblock Moments of additive functions and sieve methods.
\newblock In {\em Number theory ({N}ew {Y}ork, 1982)}, volume 1052 of {\em
  Lecture Notes in Math.}, pages 1--25. Springer, Berlin, 1984.

\bibitem{alladi1985}
K.~Alladi.
\newblock Moments of additive functions and the sequence of shifted primes.
\newblock {\em Pacific J. Math.}, 118(2):261--275, 1985.

\bibitem{alladi1122}
K.~Alladi.
\newblock A study of the moments of additive functions using {L}aplace
  transforms and sieve methods.
\newblock In {\em Number theory ({O}otacamund, 1984)}, volume 1122 of {\em
  Lecture Notes in Math.}, pages 1--37. Springer, Berlin, 1985.

\bibitem{antoniak1974}
C.~E. Antoniak.
\newblock Mixtures of {D}irichlet processes with applications to {B}ayesian
  nonparametric problems.
\newblock {\em Ann. Statist.}, 2:1152--1174, 1974.

\bibitem{arratia2003}
R.~Arratia, A.~D. Barbour, and S.~Tavar\'{e}.
\newblock {\em Logarithmic combinatorial structures: a probabilistic approach}.
\newblock EMS Monographs in Mathematics. European Mathematical Society (EMS),
  Z\"{u}rich, 2003.

\bibitem{arratia2014}
R.~Arratia, F.~Kochman, and V.~S. Miller.
\newblock Extensions of {B}illingsley's theorem via multi-intensities.
\newblock {\em arXiv preprint arXiv:1401.1555}, 2014.

\bibitem{betz2009spatial}
V.~Betz and D.~Ueltschi.
\newblock Spatial random permutations and infinite cycles.
\newblock {\em Comm. Math. Phys.}, 285(2):469--501, 2009.

\bibitem{betz2011spatial}
V.~Betz and D.~Ueltschi.
\newblock Spatial random permutations with small cycle weights.
\newblock {\em Probab. Theory Related Fields}, 149(1-2):191--222, 2011.

\bibitem{betz2011}
V.~Betz, D.~Ueltschi, and Y.~Velenik.
\newblock Random permutations with cycle weights.
\newblock {\em Ann. Appl. Probab.}, 21(1):312--331, 2011.

\bibitem{billingsley1969}
P.~Billingsley.
\newblock On the central limit theorem for the prime divisor functions.
\newblock {\em Amer. Math. Monthly}, 76:132--139, 1969.

\bibitem{billingsley1972}
P.~Billingsley.
\newblock On the distribution of large prime divisors.
\newblock {\em Period. Math. Hungar.}, 2:283--289, 1972.
\newblock Collection of articles dedicated to the memory of Alfr\'{e}d
  R\'{e}nyi, I.

\bibitem{cipriani2015}
A.~Cipriani and D.~Zeindler.
\newblock The limit shape of random permutations with polynomially growing
  cycle weights.
\newblock {\em ALEA Lat. Am. J. Probab. Math. Stat.}, 12(2):971--999, 2015.

\bibitem{de2021remarks}
R.~de~la Bret\`eche and G.~Tenenbaum.
\newblock Remarks on the {S}elberg-{D}elange method.
\newblock {\em Acta Arith.}, 200(4):349--369, 2021.

\bibitem{dereich2015}
S.~Dereich and P.~M\"{o}rters.
\newblock Cycle length distributions in random permutations with diverging
  cycle weights.
\newblock {\em Random Structures Algorithms}, 46(4):635--650, 2015.

\bibitem{donnelly1993}
P.~Donnelly and G.~Grimmett.
\newblock On the asymptotic distribution of large prime factors.
\newblock {\em J. London Math. Soc. (2)}, 47(3):395--404, 1993.

\bibitem{eberhard2016}
S.~Eberhard, K.~Ford, and B.~Green.
\newblock Permutations fixing a {$k$}-set.
\newblock {\em Int. Math. Res. Not. IMRN}, (21):6713--6731, 2016.

\bibitem{eberhard2017}
S.~Eberhard, K.~Ford, and B.~Green.
\newblock Invariable generation of the symmetric group.
\newblock {\em Duke Math. J.}, 166(8):1573--1590, 2017.

\bibitem{elboim2019}
D.~Elboim and R.~Peled.
\newblock Limit distributions for {E}uclidean random permutations.
\newblock {\em Comm. Math. Phys.}, 369(2):457--522, 2019.

\bibitem{Elliott1}
P.~D. T.~A. Elliott.
\newblock Central limit theorems for classical cusp forms.
\newblock {\em Ramanujan J.}, 36(1-2):81--98, 2015.

\bibitem{Elliott2}
P.~D. T.~A. Elliott.
\newblock Corrigendum to: ``{Central} limit theorems for classical cusp
  forms''.
\newblock {\em Ramanujan J.}, 36(1-2):99--102, 2015.

\bibitem{ercolani2014}
N.~M. Ercolani and D.~Ueltschi.
\newblock Cycle structure of random permutations with cycle weights.
\newblock {\em Random Structures Algorithms}, 44(1):109--133, 2014.

\bibitem{erdos1940}
P.~Erd\"{o}s and M.~Kac.
\newblock The {G}aussian law of errors in the theory of additive number
  theoretic functions.
\newblock {\em Amer. J. Math.}, 62:738--742, 1940.

\bibitem{erlihson2008}
M.~M. Erlihson and B.~L. Granovsky.
\newblock Limit shapes of {G}ibbs distributions on the set of integer
  partitions: the expansive case.
\newblock {\em Ann. Inst. Henri Poincar\'{e} Probab. Stat.}, 44(5):915--945,
  2008.

\bibitem{ewens1972}
W.~J. Ewens.
\newblock The sampling theory of selectively neutral alleles.
\newblock {\em Theoret. Population Biology}, 3:87--112; erratum, ibid. 3
  (1972), 240; erratum, ibid. 3 (1972), 376, 1972.

\bibitem{granville2006}
A.~Granville.
\newblock Cycle lengths in a permutation are typically {P}oisson.
\newblock {\em Electron. J. Combin.}, 13(1):Research Paper 107, 23, 2006.

\bibitem{granville2007prime}
A.~Granville.
\newblock Prime divisors are {P}oisson distributed.
\newblock {\em Int. J. Number Theory}, 3(1):1--18, 2007.

\bibitem{granville2008anatomy}
A.~Granville.
\newblock The anatomy of integers and permutations.
\newblock {\em preprint}, 2008.

\bibitem{granville2017}
A.~Granville and D.~Koukoulopoulos.
\newblock Beyond the {LSD} method for the partial sums of multiplicative
  functions.
\newblock {\em Ramanujan J.}, 49(2):287--319, 2019.

\bibitem{granville2007}
A.~Granville and K.~Soundararajan.
\newblock Sieving and the {E}rd{\H o}s-{K}ac theorem.
\newblock In {\em Equidistribution in number theory, an introduction}, volume
  237 of {\em NATO Sci. Ser. II Math. Phys. Chem.}, pages 15--27. Springer,
  Dordrecht, 2007.

\bibitem{hansen1990}
J.~C. Hansen.
\newblock A functional central limit theorem for the {E}wens sampling formula.
\newblock {\em J. Appl. Probab.}, 27(1):28--43, 1990.

\bibitem{heathbrown1988}
D.~R. Heath-Brown.
\newblock The number of primes in a short interval.
\newblock {\em J. Reine Angew. Math.}, 389:22--63, 1988.

\bibitem{hoheisel1930}
G.~{Hoheisel}.
\newblock {Primzahlprobleme in der Analysis.}
\newblock {\em {Sitzungsber. Preu{\ss}. Akad. Wiss., Phys.-Math. Kl.}},
  1930:580--588, 1930.

\bibitem{ingham1941}
A.~E. Ingham.
\newblock A {T}auberian theorem for partitions.
\newblock {\em Ann. of Math. (2)}, 42:1075--1090, 1941.

\bibitem{Khan2021}
R.~Khan, M.~B. Milinovich, and U.~Subedi.
\newblock A {W}eighted {V}ersion of the {E}rd{\H o}s-{K}ac {T}heorem.
\newblock {\em Journal of Number Theory}, 2021.

\bibitem{kingman1982}
J.~F.~C. Kingman.
\newblock The coalescent.
\newblock {\em Stochastic Process. Appl.}, 13(3):235--248, 1982.

\bibitem{lugo2009}
M.~Lugo.
\newblock Profiles of permutations.
\newblock {\em Electron. J. Combin.}, 16(1):Research Paper 99, 20, 2009.

\bibitem{Manst2002}
E.~Manstavi\v{c}ius.
\newblock Mappings on decomposable combinatorial structures: analytic approach.
\newblock {\em Combin. Probab. Comput.}, 11(1):61--78, 2002.

\bibitem{Manst2009}
E.~Manstavi\v{c}ius.
\newblock An analytic method in probabilistic combinatorics.
\newblock {\em Osaka J. Math.}, 46(1):273--290, 2009.

\bibitem{Manst2017}
E.~Manstavi\v{c}ius.
\newblock On mean values of multiplicative functions on the symmetric group.
\newblock {\em Monatsh. Math.}, 182(2):359--376, 2017.

\bibitem{maples2012}
K.~Maples, A.~Nikeghbali, and D.~Zeindler.
\newblock On the number of cycles in a random permutation.
\newblock {\em Electron. Commun. Probab.}, 17:no. 20, 13, 2012.

\bibitem{marenich1983}
E.~E. {Marenich}.
\newblock {Summation of the values of arithmetic functions.}
\newblock {\em {Math. USSR, Sb.}}, 46:51--73, 1983.

\bibitem{nikeghbali20132}
A.~Nikeghbali, J.~Storm, and D.~Zeindler.
\newblock Large cycles and a functional central limit theorem for generalized
  weighted random permutations.
\newblock {\em arXiv preprint arXiv:1302.5938}, 2013.

\bibitem{nikeghbali2013}
A.~Nikeghbali and D.~Zeindler.
\newblock The generalized weighted probability measure on the symmetric group
  and the asymptotic behavior of the cycles.
\newblock {\em Ann. Inst. Henri Poincar\'{e} Probab. Stat.}, 49(4):961--981,
  2013.

\bibitem{odoni1991}
R.~W.~K. Odoni.
\newblock A problem of {R}ankin on sums of powers of cusp-form coefficients.
\newblock {\em J. London Math. Soc. (2)}, 44(2):203--217, 1991.

\bibitem{odoni2002}
R.~W.~K. Odoni.
\newblock Solution of a generalised version of a problem of {R}ankin on sums of
  powers of cusp-form coefficients.
\newblock {\em Acta Arith.}, 104(3):201--223, 2002.

\bibitem{petrov1995}
V.~V. Petrov.
\newblock {\em Limit theorems of probability theory}, volume~4 of {\em Oxford
  Studies in Probability}.
\newblock The Clarendon Press, Oxford University Press, New York, 1995.
\newblock Sequences of independent random variables, Oxford Science
  Publications.

\bibitem{pitman1997}
J.~Pitman and M.~Yor.
\newblock The two-parameter {P}oisson-{D}irichlet distribution derived from a
  stable subordinator.
\newblock {\em Ann. Probab.}, 25(2):855--900, 1997.

\bibitem{robles2018random}
N.~Robles and D.~Zeindler.
\newblock Random permutations with logarithmic cycle weights.
\newblock {\em Ann. Inst. Henri Poincar\'{e} Probab. Stat.}, 56(3):1991--2016,
  2020.

\bibitem{schwarz1965}
W.~Schwarz.
\newblock Einige {A}nwendungen {T}auberscher {S}\"{a}tze in der
  {Z}ahlentheorie. {A}.
\newblock {\em J. Reine Angew. Math.}, 219:67--96, 1965.

\bibitem{storm2015}
J.~Storm and D.~Zeindler.
\newblock The order of large random permutations with cycle weights.
\newblock {\em Electron. J. Probab.}, 20:no. 126, 34, 2015.

\bibitem{tenenbaum2015}
G.~Tenenbaum.
\newblock {\em Introduction to analytic and probabilistic number theory},
  volume 163 of {\em Graduate Studies in Mathematics}.
\newblock American Mathematical Society, Providence, RI, third edition, 2015.
\newblock Translated from the 2008 French edition by Patrick D. F. Ion.

\bibitem{tenenbaum2017}
G.~Tenenbaum.
\newblock Moyennes effectives de fonctions multiplicatives complexes.
\newblock {\em Ramanujan J.}, 44(3):641--701, 2017.

\bibitem{timashev2008}
A.~N. Timashev.
\newblock Random permutations with cycle lengths in a given finite set.
\newblock {\em Diskret. Mat.}, 20(1):25--37, 2008.

\bibitem{watterson1976}
G.~A. Watterson.
\newblock The stationary distribution of the infinitely-many neutral alleles
  diffusion model.
\newblock {\em J. Appl. Probability}, 13(4):639--651, 1976.

\bibitem{wirsing1961}
E.~Wirsing.
\newblock Das asymptotische {V}erhalten von {S}ummen \"{u}ber multiplikative
  {F}unktionen.
\newblock {\em Math. Ann.}, 143:75--102, 1961.

\bibitem{wirsing1967}
E.~Wirsing.
\newblock Das asymptotische {V}erhalten von {S}ummen \"{u}ber multiplikative
  {F}unktionen. {II}.
\newblock {\em Acta Math. Acad. Sci. Hungar.}, 18:411--467, 1967.

\bibitem{yakymiv2007}
A.~L. Yakymiv.
\newblock Random {$A$}-permutations: convergence to a {P}oisson process.
\newblock {\em Mat. Zametki}, 81(6):939--947, 2007.

\end{thebibliography}

\Addresses
\end{document}